\documentclass[11pt]{article}

\usepackage{amssymb,amsmath,amsthm}
\usepackage{amsfonts}
\usepackage{latexsym}
\usepackage{enumerate}
\usepackage{wasysym}
\usepackage{fancyhdr}
\usepackage{lipsum}
\usepackage{epsfig}
\usepackage{breqn}
\usepackage{footnpag}
\usepackage{rotating}
\usepackage{amsfonts}
\usepackage{setspace}
\usepackage{fullpage}
\usepackage{enumitem}
\usepackage{bbold} 
\usepackage{comment}
\usepackage{pgf,tikz}
\usepackage{hyperref}
\usepackage{verbatim}
\usepackage{color}
\usepackage{mathrsfs}
\usepackage{soul}
\usetikzlibrary{arrows}
\usepackage{array}

\newtheoremstyle{dotless}{}{}{\itshape}{}{\bfseries}{}{ }{}

  \theoremstyle{dotless}

\newtheorem{theorem}{Theorem}[section] 

\newtheorem{conjecture}[theorem]{Conjecture}
\newtheorem{lemma}[theorem]{Lemma}
\newtheorem{corollary}[theorem]{Corollary}

\newtheorem{observation}[theorem]{Observation}

\newtheorem{construction}[theorem]{Construction}

\newtheorem*{theoremsp1}{Theorem \ref{final bounds}}

\newcommand{\abs}[1]{\left\lvert{#1}\right\rvert}
\newcommand{\FL}[1]{\left\lfloor #1 \right\rfloor}
\newcommand{\CL}[1]{\left\lceil #1 \right\rceil}

\DeclareMathOperator{\ex}{ex}
\DeclareMathOperator{\sat}{sat}

\title{Saturation of Berge Hypergraphs}
\author{Sean English, Nathan Graber, Pamela Kirkpatrick, \\ Abhishek Methuku, Eric C. Sullivan}

\begin{document}
\maketitle

\begin{abstract}

Given a graph $F$, a hypergraph is a Berge-$F$ if it can be obtained by expanding each edge in $F$ to a hyperedge containing it. A hypergraph $H$ is Berge-$F$-saturated if $H$ does not contain a subgraph that is a Berge-$F$, but for any edge $e\in E(\overline{H})$, $H+e$ does.
The $k$-uniform saturation number of Berge-$F$ is the minimum number of edges in a $k$-uniform Berge-$F$-saturated hypergraph on $n$ vertices. For $k=2$ this definition coincides with the classical definition of saturation for graphs. In this paper we study the saturation numbers for Berge triangles, paths, cycles, stars and matchings in $k$-uniform hypergraphs. 
\end{abstract}

\section{Introduction}~
Given simple graph $G$ and a collection of simple graphs $\mathcal{F}$, we say $G$ is $\mathcal{F}$-\textit{saturated} if $G$ does not contain any element of $\mathcal{F}$ as a subgraph, but $G+e$ contains some member of $\mathcal{F}$ as a subgraph for each $e\in E(\overline{G})$. 
The maximum possible number of edges in a graph $G$ on $n$ vertices that is $\mathcal{F}$-saturated is known as the {\it Tur\'an number} or the extremal number of $\mathcal{F}$, and is denoted $\ex(n, \mathcal{F})$.
In 1907, Mantel proved one of the first results on extremal numbers, finding them for triangles \cite{M07}. Mantel's result was generalized for all complete graphs in 1941 by Tur\'an \cite{T41}.

On the other end of the spectrum, the minimum number of edges of an $\mathcal{F}$-saturated graph on $n$ vertices is known as the {\it saturation number} of $\mathcal{F}$ and is denoted $\sat(n, \mathcal{F})$. When $\mathcal{F}=\{F\}$ contains only a single graph, we write $\sat(n,F)$ for convenience. Saturation numbers for graphs were first studied by Erd\H{o}s, Hajnal and Moon in \cite{EHM64}, where they determined $\sat(n,K_m)$ and the $K_m$-saturated graphs that achieve that saturation number.
In \cite{KT_minedge_1986}, K\'aszonyi and Tuza determine the saturation number for paths, stars and matchings, and provide a general upper bound.

\begin{theorem}[K\'aszonyi and Tuza]\mbox{}
	\begin{enumerate} 
		\item Let 
		\[
		a_m=\begin{cases}
		3\cdot 2^{t-1}-2 &\text{if }  m=2t \\
		4\cdot 2^{t-1}-2 &\text{if }  m=2t+1.
		\end{cases}
		\]
		If $n\geq a_m$ and $k\geq 6,$ then $\sat(n, P_m)=n-\left\lfloor\frac{n}{a_m}\right\rfloor$.
		\item Let $S_t=K_{1, t-1}$ denote a star on $t$ vertices. Then, 
		\[
		\sat(n, S_t)=\begin{cases}
		\binom{t-1}{2}+\binom{n-t+1}{2} &\text{if }t\leq n\leq \frac{3t-3}{2} \\[0.25em]
		\left\lceil\frac{(t-2)n}{2}-\frac{(t-1)^2}{8}\right\rceil &\text{if } \frac{3t-3}{2}\leq n
		\end{cases}
		\]
		\item For $n\geq 3t-3$, $\sat(n, tK_2)=3t-3.$
\end{enumerate}\end{theorem}

The concept of saturation has been extended to hypergraphs, a generalization of graphs in which an edge can contain any number of vertices. A hypergraph is called $k$-{\it uniform} if every edge contains exactly $k$ vertices. Note that a $2$-uniform hypergraph is simply a graph. Unless otherwise stated, in this paper we will assume all hypergraphs are $k$-uniform. Given a collection of $k$-uniform hypergraphs $\mathcal{F}$, one can define $\mathcal{F}$-saturated and $\sat_k(n,\mathcal{F})$ analogously to the graph case.

The classical definition of a hypergraph cycle due to Berge is the following. A Berge cycle of length $k$ is an alternating sequence of distinct vertices and edges of the form $v_1$,$e_{1}$,$v_2$,$e_{2},\ldots,v_k$,$e_{k}$,$v_1$ where $v_i,v_{i+1} \in e_{i}$ for each $1\leq i\leq k-1$ and $v_k,v_1 \in e_{k}$ and is denoted Berge-$C_k$. A Berge-path is defined similarly. In this paper, whenever we refer to a path in a hypergraph, unless otherwise stated, this will refer to a Berge path.

Gerbner and Palmer \cite{GP15} gave the following generalization of the definitions of Berge cycles and Berge paths. Let $F$ be a graph and $H$ a hypergraph. We say $H$ is a \textit{Berge-F} if there is a bijection $\phi:E(F) \rightarrow E(H)$ such that $e \subseteq \phi(e)$ for all $e \in E(F)$. This can be thought of as expanding each edge of $F$ to an edge of $H$ or shrinking each edge of $H$ down to an edge of $F$.
For a graph $F$ we denote the set of all $k$-uniform hypergraphs that are a Berge-$F$ by $\mathcal{B}_k(F)$. Note that $\mathcal{B}_k(F)$ is finite for any graph $F$.
We say a hypergraph $H$ \textit{contains} a Berge-$F$ if $H$ contains a subhypergraph that is a Berge-$F$, and that $H$ is Berge-$F$-free otherwise. 
For example, the hypergraph in Figure \ref{bsatgraphex} contains a Berge-$P_5$, but is Berge-$P_6$-free. 

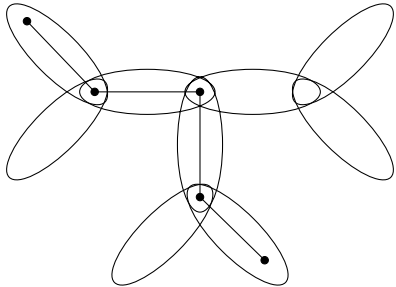
\begin{figure}[h]
	\begin{center}
		\begin{tikzpicture}[scale=0.4]
		\draw
		(-1.75,0)ellipse (2.25cm and .75cm)
		(1.75,0)ellipse (2.25cm and .75cm)
		(0,-1.75)ellipse (.75cm and 2.25cm);
		\draw[rotate around={-45:(-4.75,1.25)}] (-4.75,1.25) ellipse (2.25cm and .75cm);
		\draw[rotate around={45:(-4.75,-1.25)}] (-4.75,-1.25) ellipse (2.25cm and .75cm);
		\draw[rotate around={-135:(4.75,1.25)}] (4.75,1.25) ellipse (2.25cm and .75cm);
		\draw[rotate around={135:(4.75,-1.25)}] (4.75,-1.25) ellipse (2.25cm and .75cm);
		\draw[rotate around={135:(1.25,-4.75)}] (1.25,-4.75) ellipse (2.25cm and .75cm);
		\draw[rotate around={45:(-1.25,-4.75)}] (-1.25,-4.75) ellipse (2.25cm and .75cm);
		\draw[fill=black] (0,0) circle (3.5pt);
		\draw[fill=black] (-3.5,0) circle (3.5pt);
		\draw[fill=black] (-5.75,2.35) circle (3.5pt);
		\draw[fill=black] (0,-3.5) circle (3.5pt);
		\draw[fill=black] (2.15,-5.6) circle (3.5pt);
		\draw(0,0)--(0,-3.5);
		\draw(-3.5,0)--(0,0);
		\draw(0,-3.5)--(2.15,-5.6);
		\draw(-3.5,0)--(-5.75,2.35);
		\end{tikzpicture}
		\caption{A hypergraph $H$ that contains a Berge-$P_5$ and is Berge-$P_6$-free.}
		\label{bsatgraphex}
	\end{center}
\end{figure}

Tur\'an-type  extremal problems for hypergraphs in the Berge sense have attracted considerable attention \cite{lazebnik2003hypergraphs, gyHori2006triangle, bollobas2008pentagons, gyHori20123, furedi20173, timmons2016r, ergemlidze2017asymptotics, gerbner2017asymptotics, palmer2017tur},  where the goal is to determine the maximum possible number of edges in a Berge-$F$-free hypergraph, called the Tur\'an number of Berge-$F$. In this paper we study saturation for Berge hypergraphs.

For the sake of notation, we wll write $\sat_k(n,\text{Berge-}F)$ instead of $\sat_k(n,\mathcal{B}_k(F))$. In \cite{pikhurko1999minimum}, Pikhurko showed that $\sat_k(n,\mathcal{F})=O(n^{k-1})$ for any finite family of $k$-uniform hypergraphs $\mathcal{F}$, so we have that  $\sat_k(n,\text{Berge-}F)=O(n^{k-1})$.

Tur\'an numbers for Berge cycles and Berge paths were found by Gy{\H{o}}ri et al. \cite{GKL10, davoodi2016erd}.  
For general Berge graphs, Tur\'an numbers were determined by Gerbner and Palmer \cite{GP15}. They showed that if $H$ is Berge-$F$-free where $|e|\ge|V(F)|$ for all $e\in E(H)$, then $\abs{E(H)} \le  ex(n,F) $. 
Since any $k$-uniform hypergraph that is Berge-$F$-saturated is also Berge-$F$-free, we see that $\sat_k(n,\text{Berge-}F) \le  ex(n,F) = O(n^2)$ when $k\geq |V(F)|$. 

Despite intensive research concerning  Tur\'an numbers for Berge hypergraphs, no work has been done for saturation numbers for Berge hypergraphs. Determining the minimum number of edges that a hypergraph can have and be Berge-$F$-saturated is different from classical saturation in several ways. In some cases, the $k$-uniform hypergraph obtained by adding $k-2$ new vertices to each edge of a minimally $F$-saturated graph is also Berge-$F$-saturated but not minimal, and in other cases, it is not even Berge-$F$-saturated. For example, the five cycle $C_5$ is $K_3$-saturated, however if we add new vertices to each edge of $C_5$ to form a $3$-uniform hypergraph, then it is not Berge-$K_3$-saturated. Also, if one adds $k-2$ new vertices to each edge in the $P_m$-saturated graph given by K\'aszonyi and Tuza in \cite{KT_minedge_1986} to form a hypergraph, the resulting hypergraph will be Berge-$P_m$-saturated, but not minimal. 

We will explore the saturation number in the Berge sense for many classes of graphs. In Section \ref{section other saturation numbers}, we determine the saturation numbers for Berge triangles, cycles, matchings, and stars. The results of this section are summarized in the following theorem. It is worth noting that the results that only contain an upper bound may not be tight, but they establish linearity in each case.

For each of the following, let $k\geq 3$, $\ell\geq 1$ and $m\geq 4$.

\begin{table}[!htbp]
\begin{tabular}{ll}
		Theorem \ref{thm:matchings}:&For all $n\geq k(\ell-1)$,
		$\sat_k(n,\text{Berge-}\ell K_2)=\ell-1$.\\[0.3cm]
		Theorem \ref{thm:triangle}:&For all $n\geq k+1$, $\sat_k(n,\text{Berge-}K_3)=\CL{\frac{n-1}{k-1}}$.\\[0.3cm]
		Theorem \ref{thm:cycle}:& For all $k\geq m-1$ and $n>m(k-(m-2))+(m-2)$,
		$\sat_k(n,\text{Berge-}C_m)\leq \CL{\frac{n-m+2}{k-m+2}}$.\\[0.3cm]
		&If $k=m-2$ and $n\geq m^2$,
		$\sat_k(n,\text{Berge-}C_m)\leq\FL{\frac{n-1}{m-2}}\binom{m-1}{k}+\frac{(n-1)\mod{(m-2)}}{k-2}$.\\[0.3cm]
		&For all $k\leq m-3$, $\ell=\max\{m/2+1,k+1\}$ and $n\geq \ell^2$,\\[0.3cm]
		&$\sat_k(n,\text{Berge-}C_m)\leq\FL{\frac{n-1}{\ell-1}}\binom{\ell}{k}+\left((n-1)\mod{(\ell-1)}\right)\binom{\ell}{k-1}$.\\[0.3cm]
		Theorem \ref{thm:k+1star}:&For all $n\geq k^2$, 
		$\sat_k(n,\text{Berge-}K_{1,k+1})=n-(k-1)$.\\[0.3cm]
		Theorem \ref{thm:mstar}:&For all $k\leq m-1$, $\sat_k(n,\text{Berge-}K_{1,m})\leq \CL{\frac{n}{m}}\binom{m}{k}$.\\
\end{tabular}
\end{table}


Our main result, however, determines $\sat_k(n,\text{Berge-}P_m)$ where $P_m$ is the simple graph path on $m$ vertices. Let
\[
a^{(k)}_m=\min\{|E(T)|\mid T\text{ is a }k\text{-uniform Berge-}P_m\text{-saturated linear tree on at least $k+1$ vertices}\}.
\]
In Sections \ref{section path saturation numbers}, \ref{section constructions} and \ref{section lower bound} we determine the values of $a^{(k)}_m$ for all $k\geq 3$, $k\neq 5$ and all $m\geq 10$ and also establish the following theorem.

\begin{theoremsp1}
	Let $k\geq 3$ with $k\neq 5$. Let $m\geq 10$, and $n\geq (k-1)a^{(k)}_m+k-1$. Then
\[
\frac{1}{k-1}\left(n-\left\lfloor\frac{n-k+2}{(k-1)a^{(k)}_m+1}\right\rfloor-k+2\right)\leq \sat_k(n,\text{Berge-}P_m)\leq\left\lceil\frac{1}{k-1}\left(n-\left\lfloor\frac{n}{(k-1)a^{(k)}_m+1}\right\rfloor\right)\right\rceil.
\]
\end{theoremsp1}
It should be noted that the bounds given here differ by at most three. The case of uniformity $k=5$ for paths has some complications not present for other values of $k$, so this case will not be covered here.

\section{Saturation Numbers for Berge Paths}\label{section path saturation numbers}

\subsection{A Lower Bound for sat$\mathbf{_k(n,\text{Berge-}P_m)}$}

We give a lower bound for $\sat_k(n,\text{Berge-}P_m)$ in terms of $a^{(k)}_m$.
\begin{theorem}\label{main theorem}
	Let $k\geq 3$, $m\geq 10$ and $n\geq (k-1)a^{(k)}_m+k-1$. Then
	\[
	\sat_k(n,\text{Berge-}P_m)\geq \frac{1}{k-1}\left(n-\left\lfloor\frac{n-k+2}{(k-1)a^{(k)}_m+1}\right\rfloor-k+2\right).
	\]
\end{theorem}

\begin{proof}

	Let $H_0$ be a minimal $k$-uniform Berge-$P_m$-saturated hypergraph on $n$ vertices. Observe that $H_0$ cannot have more than $k-1$ isolated vertices.

	\noindent \textbf{Case 1.} $H_0$ contains a single edge, $e$ as a component.
	
	If $H_0$ contains another component, say $T$ that is a linear tree, then there must be a path of length at least $m-2$ starting at each vertex of $T$, since otherwise we can add an edge containing that vertex and $k-1$ vertices in $e$ without creating a Berge-$P_m$. Then if we consider a vertex in the center of $T$, we have that this vertex has a path of length at least $m-2$ away from it, and since it is in the center, a second path of at least length $m-3$ away from it, and these paths can share at most one edge, so putting these two together, we have a path of length at least $2m-6> m-1$, so $T$ contains a Berge-$P_m$, contradicting saturation. Thus, $e$ is the only component of $H_0$ that is a linear tree. Since any $k$-uniform hypergraph on $n_0$ vertices, no component of which is a linear tree, has at least $\frac{n_0}{k-1}$ edges, we have

	
	\[
		|E(H_0)| \geq \frac{n-k}{k-1}+1=\frac{n-1}{k-1} >\frac{1}{k-1}\left(n-\left\lfloor\frac{n-k+2}{(k-1)a_m^{(k)}+1}\right\rfloor-k+2\right). 
	\]
	
	\noindent \textbf{Case 2.} $H_0$ has no isolated edges and at least $k-1$ isolated vertices.
	
	If $H_0$ has a leaf, then we can add an edge containing the $k-1$ isolated vertices and the vertex of degree greater than 1 in the leaf without creating a Berge-$P_m$. Thus $H_0$ would not be saturated. Then every component in $H_0$ that is not an isolated vertex cannot be a linear tree since any non-trivial linear tree contains a leaf. Again, since any $k$-uniform hypergraph on $n_0$ vertices, no component of which is a linear tree, has at least $\frac{n_0}{k-1}$ edges and since $n\geq (k-1)a^{(k)}_m+k-1$, we have
	\[
	|E(H_0)|\geq\frac{n-(k-1)}{k-1}\geq\frac{1}{k-1}\left(n-\left\lfloor\frac{n-k+2}{(k-1)a_m^{(k)}+1}\right\rfloor-k+2\right).
	\]
	
	\noindent \textbf{Case 3.} $H_0$ has no isolated edges, and no more than $k-2$ isolated vertices.
	
	Assume that $H_0$ has $c\geq 1$ connected components $C_1,\dots,C_c$ and assume without loss of generality that the first $t$ of these are linear trees for some $0\leq t\leq c$. If a linear tree has $a^{(k)}_m$ edges, then it has $b=(k-1)a^{(k)}_m+1$ vertices. There are at most $k-2$ isolated vertices in $H_0$, so there are at least $t-k+2$ non-trivial trees in $H_0$. This implies that  $(t-k+2) b\leq n-k+2$, or $t\leq \lfloor\frac{n-k+2}b\rfloor+k-2$. For $i\leq t$, $|E(C_i)|=\frac{|V(C_i)|-1}{k-1}$, and for $i>t$, $|E(C_i)|\geq\frac{|V(C_i)|}{k-1}$. Thus
	\[
	|E(H_0)|\geq \sum_{i=1}^t \frac{|V(C_i)|-1}{k-1} + \sum_{i=t+1}^c \frac{|V(C_i)|}{k-1}=\frac{1}{k-1}(n-t)\geq \frac{1}{k-1}\left(n-\left\lfloor\frac{n-k+2}{b}\right\rfloor-k+2\right).
	\]
\end{proof}

In Section $\ref{section constructions}$ we will give constructions for $k$-uniform Berge-$P_m$-saturated linear trees and in Section $\ref{section lower bound}$ we will show that these constructions are minimal. The results of these sections will imply the following theorem.

\begin{theorem}\label{a^k_m theorem}
	Let $m\geq 10$. If $m=4s+r$ for $1\leq r\leq 4$, then
	\[
	a^{(3)}_m=(3+r)2^s-5.
	\]
	If $m=6s+r$ for $0\leq r\leq 5$, then
	\[
	a^{(4)}_m=(6+r)2^s-8.
	\]
	If $k\geq 6$, then
	\[
	a^{(k)}_m=
	\begin{cases} 
	2^{s+1}+2^s+2^{s-1}+2^{s-2}-6&\text{ if }m=4s,\\
	2^{s+2}+2^{s-1}-6&\text{ if }m=4s+1,\\
	2^{s+2}+2^{s}-6&\text{ if }m=4s+2,\\
	2^{s+2}+2^{s+1}+2^{s-1}-6&\text{ if }m=4s+3.
	\end{cases}
	\]
\end{theorem}

\begin{proof}
	Theorems \ref{k=3 P_{m-1} final count theorem}, \ref{k=4 P_{m-1} final count theorem} and \ref{k>=6 P_{m-1} final count theorem} give us the lower bounds on $a^{(k)}_m$, while Lemma \ref{tree saturation lemma} and observations \ref{k=3 tree construction edge count observation}, \ref{k=4 tree construction edge count observation} and \ref{k>=6 tree construction edge count observation} give us the matching upper bounds.
\end{proof}

Theorems \ref{main theorem} and \ref{a^k_m theorem} give us a lower bound on $\sat_k(n,\text{Berge-}P_m)$, while Construction \ref{H construction}, Observation \ref{H edge count observation} and Lemma \ref{H saturation lemma} give us an upper bound. These bounds are summarized in the following theorem.

\begin{theorem}\label{final bounds}
	Let $k\geq 3$ with $k\neq 5$. Let $m\geq 10$, and $n\geq (k-1)a^{(k)}_m+k-1$. Then
	\[
	\frac{1}{k-1}\left(n-\left\lfloor\frac{n-k+2}{(k-1)a^{(k)}_m+1}\right\rfloor-k+2\right)\leq \sat_k(n,\text{Berge-}P_m)\leq\left\lceil\frac{1}{k-1}\left(n-\left\lfloor\frac{n}{(k-1)a^{(k)}_m+1}\right\rfloor\right)\right\rceil.
	\]
\end{theorem}
It is worth noting that the upper and lower bounds provided here differ by at most three. 
In the proof of Theorem \ref{main theorem}, we allow for the possibility that a minimal, saturated graph contains up to $k-2$ isolated vertices, but this may not be possible. If it can be shown that there are no isolated vertices in the minimal construction, the upper and lower bounds will match. This leads to the following conjecture:

\begin{conjecture}
	Let $k\geq 3$, $m\geq 10$, $n\geq (k-1)a^{(k)}_m+1$. Then
	\[
	\sat_k(n,\text{Berge-}P_m)=\left\lceil\frac{1}{k-1}\left(n-\left\lfloor\frac{n}{(k-1)a^{(k)}_m+1}\right\rfloor\right)\right\rceil.
	\]
\end{conjecture}

\section{Constructing an Upper Bound on $\mathbf{a^{(k)}_m}$}\label{section constructions}
In this section, we give a construction for a Berge-$P_m$ saturated linear tree, $T^{(k)}_m$, for $k=3,4$ and $k\geq 6$ and for all $m\geq 10$, and we show how to use $T^{(k)}_m$ to create a saturated graph on $n$ vertices for any $n\geq (k-1)a^{(k)}_m+1$.
In Section~\ref{section lower bound} we will show that each $T^{(k)}_m$ is minimal, or more precisely, $\left|E\left(T^{\left(k\right)}_m\right)\right|=a^{(k)}_m$.

\subsection{Constructing $\mathbf{T^{(k)}_m}$}
We begin determining the upper bound for $\sat_3(n, \text{Berge-}P_m)$ by constructing 3-uniform linear trees that are Berge-$P_m$ saturated for $m\ge 8$. 

\begin{construction}\label{con:k=3}
For $m\geq8$, let $T^{(3)}_m$ be the linear tree built as follows:
\begin{enumerate}
\item Start with a central vertex $v$ in an edge, $e$. 
\item Add a pendant edge to both vertices in $e-\{v\}$ to begin the two main branches of the tree. Call these the left and right sides of the tree.
\item To each vertex added in the previous step, append a pendant edge. 
\item Add a pendant edge to one degree 1 vertex in each of the edges added in step 3. 
\item If $m$ is odd, repeat steps 3 and 4 until there are $\frac{m-3}{2}$ levels of the linear tree after the initial edge $e$. If $m$ is even, repeat steps 3 and 4 until there are $\frac{m-4}{2}$ levels after the initial edge $e$, and then add an additional level to the left side of the tree.
\end{enumerate}
\end{construction}

Figures \ref{m8pathconstruction} and \ref{m9pathconstruction} show the constructions for $m=8$ and $m=11$ respectively.

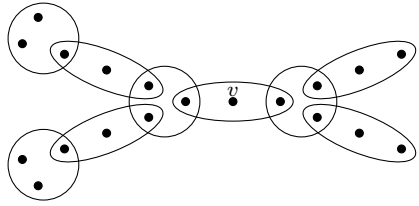
\begin{figure}[ht]
\begin{center}
\begin{tikzpicture}[line cap=round,line join=round,>=triangle 45,x=.7cm,y=.7cm]
	\clip(-6.1,1.5) rectangle (6.1,5.5);
	\draw [rotate around={0.:(0,3.6)}] (0,3.6) ellipse (.8cm and 0.26cm);
	\draw(-1.3,3.6) circle (0.47cm);
	\draw(1.3,3.6) circle (0.47cm);
	\draw [rotate around={-21.8:(-2.4,4.2)}] (-2.4,4.2) ellipse (.8cm and 0.26cm);
	\draw [rotate around={21.8:(-2.4,3.)}] (-2.4,3.) ellipse (.8cm and 0.26cm);
	\draw [rotate around={21.8:(2.4,4.2)}] (2.4,4.2) ellipse (.8cm and 0.26cm);
	\draw [rotate around={-21.8:(2.4,3.)}] (2.4,3.) ellipse (.8cm and 0.26cm);
	\draw(-3.6,4.8) circle (0.47cm);
	\draw(-3.6,2.4) circle (0.47cm);
	\begin{scriptsize}
	\draw [fill=black] (-1.6,3.9) circle (1.5pt);
	\draw [fill=black] (-1.6,3.3) circle (1.5pt);
	\draw [fill=black] (1.6,3.9) circle (1.5pt);
	\draw [fill=black] (1.6,3.3) circle (1.5pt);
	\draw [fill=black] (-3.2,4.5) circle (1.5pt);
	\draw [fill=black] (-3.2,2.7) circle (1.5pt);
	\draw [fill=black] (3.2,4.5) circle (1.5pt);
	\draw [fill=black] (3.2,2.7) circle (1.5pt);
	\draw [fill=black] (0.,3.6) circle (1.5pt);
	\draw[color=black] (0,3.8) node {$v$};
	\draw [fill=black] (-0.9,3.6) circle (1.5pt);
	\draw [fill=black] (0.9,3.6) circle (1.5pt);
	\draw [fill=black] (2.4,4.2) circle (1.5pt);
	\draw [fill=black] (2.4,3.) circle (1.5pt);
	\draw [fill=black] (-2.4,4.2) circle (1.5pt);
	\draw [fill=black] (-2.4,3.) circle (1.5pt);
	\draw [fill=black] (-4,2.5) circle (1.5pt);
	\draw [fill=black] (-3.7,2) circle (1.5pt);
	\draw [fill=black] (-4.,4.7) circle (1.5pt);
	\draw [fill=black] (-3.7,5.2) circle (1.5pt);
	\end{scriptsize}
\end{tikzpicture}
\caption{A copy of $T^{(3)}_8$, as described in Construction~\ref{con:k=3}.}
\label{m8pathconstruction}
\end{center}
\end{figure}

\begin{figure}[ht]
\begin{center}
\begin{tikzpicture}[line cap=round,line join=round,>=triangle 45,x=.7cm,y=.7cm]
	\clip(-6.1,0.5) rectangle (6.1,6.5);
	\draw [rotate around={0.:(0,3.6)}] (0,3.6) ellipse (.8cm and 0.26cm);
	\draw(-1.3,3.6) circle (0.47cm);
	\draw(1.3,3.6) circle (0.47cm);
	\draw [rotate around={-21.8:(-2.4,4.2)}] (-2.4,4.2) ellipse (.8cm and 0.26cm);
	\draw [rotate around={21.8:(-2.4,3.)}] (-2.4,3.) ellipse (.8cm and 0.26cm);
	\draw [rotate around={21.8:(2.4,4.2)}] (2.4,4.2) ellipse (.8cm and 0.26cm);
	\draw [rotate around={-21.8:(2.4,3.)}] (2.4,3.) ellipse (.8cm and 0.26cm);
	\draw(-3.6,4.8) circle (0.47cm);
	\draw(-3.6,2.4) circle (0.47cm);
	\draw(3.6,2.4) circle (0.47cm);
	\draw(3.6,4.8) circle (0.47cm);
	\draw [rotate around={-5:(-4.9,2.6)}] (-4.9,2.6) ellipse (.8cm and 0.26cm);
	\draw [rotate around={21.8:(-4.5,1.65)}] (-4.5,1.65) ellipse (.8cm and 0.26cm);
	\draw [rotate around={5:(4.9,2.6)}] (4.9,2.6) ellipse (.8cm and 0.26cm);
	\draw [rotate around={-21.8:(4.5,1.65)}] (4.5,1.65) ellipse (.8cm and 0.26cm);
	\draw [rotate around={5:(-4.9,4.6)}] (-4.9,4.6) ellipse (.8cm and 0.26cm);
	\draw [rotate around={-21.8:(-4.5,5.55)}] (-4.5,5.55) ellipse (.8cm and 0.26cm);
	\draw [rotate around={-5:(4.9,4.6)}] (4.9,4.6) ellipse (.8cm and 0.26cm);
	\draw [rotate around={21.8:(4.5,5.55)}] (4.5,5.55) ellipse (.8cm and 0.26cm);
	\begin{scriptsize}
	\draw [fill=black] (-1.6,3.9) circle (1.5pt);
	\draw [fill=black] (-1.6,3.3) circle (1.5pt);
	\draw [fill=black] (1.6,3.9) circle (1.5pt);
	\draw [fill=black] (1.6,3.3) circle (1.5pt);
	\draw [fill=black] (-3.2,4.5) circle (1.5pt);
	\draw [fill=black] (-3.2,2.7) circle (1.5pt);
	\draw [fill=black] (3.2,4.5) circle (1.5pt);
	\draw [fill=black] (3.2,2.7) circle (1.5pt);
	\draw [fill=black] (0.,3.6) circle (1.5pt);
	\draw[color=black] (0,3.8) node {$v$};
	\draw [fill=black] (-0.9,3.6) circle (1.5pt);
	\draw [fill=black] (0.9,3.6) circle (1.5pt);
	\draw [fill=black] (2.4,4.2) circle (1.5pt);
	\draw [fill=black] (2.4,3.) circle (1.5pt);
	\draw [fill=black] (-2.4,4.2) circle (1.5pt);
	\draw [fill=black] (-2.4,3.) circle (1.5pt);
	
	\draw [fill=black] (-4,2.5) circle (1.5pt);
	\draw [fill=black] (-3.7,2) circle (1.5pt);
	\draw [fill=black] (-4.85,2.6) circle (1.5pt);
	\draw [fill=black] (-5.6,2.7) circle (1.5pt);
	\draw [fill=black] (-4.45,1.65) circle (1.5pt);
	\draw [fill=black] (-5,1.45) circle (1.5pt);
	
	\draw [fill=black] (4.,2.5) circle (1.5pt);
	\draw [fill=black] (3.7,2.) circle (1.5pt);
	\draw [fill=black] (4.85,2.6) circle (1.5pt);
	\draw [fill=black] (5.6,2.7) circle (1.5pt);
	\draw [fill=black] (4.45,1.65) circle (1.5pt);
	\draw [fill=black] (5,1.45) circle (1.5pt);
	
	\draw [fill=black] (4.,4.7) circle (1.5pt);
	\draw [fill=black] (3.7,5.2) circle (1.5pt);
	\draw [fill=black] (4.85,4.6) circle (1.5pt);
	\draw [fill=black] (5.6,4.5) circle (1.5pt);
	\draw [fill=black] (4.45,5.55) circle (1.5pt);
	\draw [fill=black] (5,5.75) circle (1.5pt);
	
	\draw [fill=black] (-4.,4.7) circle (1.5pt);
	\draw [fill=black] (-3.7,5.2) circle (1.5pt);
	\draw [fill=black] (-4.85,4.6) circle (1.5pt);
	\draw [fill=black] (-5.6,4.5) circle (1.5pt);
	\draw [fill=black] (-4.45,5.55) circle (1.5pt);
	\draw [fill=black] (-5,5.75) circle (1.5pt);

	\end{scriptsize}
\end{tikzpicture}
\caption{A copy of $T^{(3)}_{11}$, as described in Construction~\ref{con:k=3}.}
\label{m9pathconstruction}
\end{center}
\end{figure}
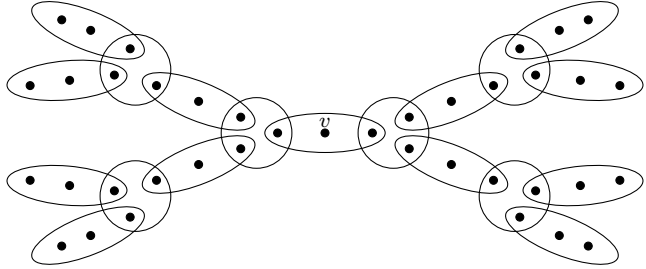

\begin{observation}\label{k=3 tree construction edge count observation}
	If $m\geq 8$, $1\leq r\leq 4$ and $m=4s+r$, then
	\[
\left|E\left(T^{(3)}_m\right)\right|=(3+r)2^{s}-5.
    \]
\end{observation}

Next, we provide the tree construction for uniformity $k=4$
\begin{construction}\label{con:k=4}
For $m\ge 10$, let $T^{(4)}_m$ be the linear tree built as follows:
\begin{enumerate}
\item Start with a central vertex, $v$ and two edges that intersect only at $v$.
\item Add a pendent edge to one degree 1 vertex in each of the edges added in the previous step to begin the two main branches of the tree. 
Call these the left and right sides of the tree. 
\item Add a pendant edge to each of two degree 1 vertices in each of the edges added in the previous step. 
\item Add a pendant edge to one degree 1 vertex in each of the edges added in the previous step.
\item Repeat steps 2,3,4 until the tree has $3\FL{\frac{m}{6}}-1$ levels after the initial vertex. 

For $m=0 \mod{6}$, we do not need any additional steps.

For $m=1 \mod{6}$, repeat step 3 on the left side of the tree.

For $m=2 \mod{6}$, repeat step 3.

For $m=3 \mod{6}$, repeat step 3, and repeat step 4 on the left side of the tree.

For $m=4 \mod{6}$, repeat steps 3 and 4.

For $m=5 \mod{6}$, repeat steps 3 and 4, and step 4 again on the left side of the tree.
\end{enumerate}
\end{construction}

See Figure \ref{fig:k4construction} for an example of the construction for $m=18$.

\begin{figure}
\begin{center}
\begin{tikzpicture}[scale=1]
\clip(-3.48,-2.5) rectangle (4.4,1.8);
\draw (-.2,0)circle[x radius =0.316cm,y radius=0.1cm,rotate=0];
\draw (.2,0) circle [x radius=.316cm, y radius=.1cm, rotate=0];

\draw (0.6,0.) circle (0.2cm);

\draw (.943,.275) circle [x radius=.317cm, y radius=.1406cm, rotate=38.6598];
\draw (1.25446,0.523614) circle [x radius=.317cm, y radius=.1406cm, rotate=38.6598];
\draw (1.5627,0.77) circle (0.20124cm);
\draw (1.8518,1.05105) circle [x radius=.317cm, y radius=.1406cm, rotate=29.6148];
\draw (2.197,1.247) circle [x radius=.317cm, y radius=.1406cm, rotate=29.6148];
\draw (2.54599,1.44564) circle [x radius=.317cm, y radius=.1406cm, rotate=29.6148];
\draw (2.006,0.75369) circle [x radius=.317cm, y radius=.1406cm, rotate=-2.10165];
\draw (2.4086,0.7389) circle [x radius=.317cm, y radius=.1406cm, rotate=-2.10165];
\draw (2.81307,0.72407) circle [x radius=.317cm, y radius=.1406cm, rotate=-2.10165];

\draw (0.943198,-0.2746) circle [x radius=.317cm, y radius=.1406cm, rotate=-38.6598];
\draw (1.25446,-0.523614) circle [x radius=.317cm, y radius=.1406cm, rotate=-38.6598];
\draw (1.5627,-0.77) circle (0.20124cm);
\draw (1.8518,-1.05105) circle [x radius=.317cm, y radius=.1406cm, rotate=-29.6148];
\draw (2.197,-1.247) circle [x radius=.317cm, y radius=.1406cm, rotate=-29.6148];
\draw (2.54599,-1.44564) circle [x radius=.317cm, y radius=.1406cm, rotate=-29.6148];
\draw (2.006,-0.75369) circle [x radius=.317cm, y radius=.1406cm, rotate=2.10165];
\draw (2.4086,-0.7389) circle [x radius=.317cm, y radius=.1406cm, rotate=2.10165];
\draw (2.81307,-0.72407) circle [x radius=.317cm, y radius=.1406cm, rotate=2.10165];

\draw (-0.6,0.) circle (0.2cm);

\draw (-0.943198,0.2746) circle [x radius=.317cm, y radius=.1406cm, rotate=-38.6598];
\draw (-1.25446,0.523614) circle [x radius=.317cm, y radius=.1406cm, rotate=-38.6598];
\draw (-1.5627,0.77) circle (0.20124cm);
\draw (-1.8518,1.05105) circle [x radius=.317cm, y radius=.1406cm, rotate=-29.6148];
\draw (-2.197,1.247) circle [x radius=.317cm, y radius=.1406cm, rotate=-29.6148];
\draw (-2.54599,1.44564) circle [x radius=.317cm, y radius=.1406cm, rotate=-29.6148];
\draw (-2.006,0.75369) circle [x radius=.317cm, y radius=.1406cm, rotate=2.10165];
\draw (-2.4086,0.7389) circle [x radius=.317cm, y radius=.1406cm, rotate=2.10165];
\draw (-2.81307,0.72407) circle [x radius=.317cm, y radius=.1406cm, rotate=2.10165];

\draw (-0.943198,-0.2746) circle [x radius=.317cm, y radius=.1406cm, rotate=38.6598];
\draw (-1.25446,-0.523614) circle [x radius=.317cm, y radius=.1406cm, rotate=38.6598];
\draw (-1.5627,-0.77) circle (0.20124cm);
\draw (-1.8518,-1.05105) circle [x radius=.317cm, y radius=.1406cm, rotate=29.6148];
\draw (-2.197,-1.247) circle [x radius=.317cm, y radius=.1406cm, rotate=29.6148];
\draw (-2.54599,-1.44564) circle [x radius=.317cm, y radius=.1406cm, rotate=29.6148];
\draw (-2.006,-0.75369) circle [x radius=.317cm, y radius=.1406cm, rotate=-2.10165];
\draw (-2.4086,-0.7389) circle [x radius=.317cm, y radius=.1406cm, rotate=-2.10165];
\draw (-2.81307,-0.72407) circle [x radius=.317cm, y radius=.1406cm, rotate=-2.10165];

\end{tikzpicture}
\caption{A copy of $T^{(4)}_{18}$, as described in Construction~\ref{con:k=4}.}\label{fig:k4construction}
\end{center}
\end{figure}
\begin{observation}\label{k=4 tree construction edge count observation}
For $m\geq 10$, let $0\leq r\leq 5$ such that $m=6s+r$. Then 
\[
\left|E\left(T^{(4)}_m\right)\right|=(6+r)2^s-8.
\]
\end{observation}

Finally, we present a construction for $T^{(k)}_m$ with $k\geq 6$.  The construction for when $m$ is divisible by $4$ is distinct from other divisibilities. We give the more general construction first.

\begin{construction}\label{con:k=k}
For $m\geq 8$ and $k\geq 6$, let $T^{(k)}_m$ be the $k$-uniform linear tree built as follows:

For $m=1,2,3 \mod 4$:
\begin{enumerate}
\item Start with a central vertex, $v$ and two edges that intersect only at $v$ that begin the two main branches of the tree. Call these the left and right sides.
\item Add two pendant edges to one degree 1 vertex in each of the edges added in the previous step. 
\item Add a pendant edge to one degree one vertex in each of the edges added in the previous step. 
\item For $m=1 \mod 4$, repeat steps 2 and 3 until there are $\frac{m-5}{2}+1$ levels after $v$, and repeat step 3 for one of the two main branches. \\For $m=2 \mod 4$, repeat steps 2 and 3 until there are $\frac{m-6}{2}+1$ levels after $v$ and then repeat step 3. \\For $m=3 \mod 4$, repeat steps 2 and 3 until there are $\frac{m-7}{2}+1$ levels after $v$, and repeat steps 2 and then 3 for one of the two main branches and just step 3 for the other.
\end{enumerate}

For $m=4s$:
\begin{enumerate}
\item Start with a central vertex, $v$ and attach three edges to create a star. These will begin the three main branches of the construction.
\item Add a pendant edge to one degree 1 vertex in each of the edges added in the previous step.
\item Add two pendant edges to one degree 1 vertex in each of the edges added in the previous step.
\item Repeat steps 2 and 3 until there are $\frac{m-4}{2}$ levels after the initial vertex and then repeat step 2. Note: For $m=8$, the construction does not use step 3. 
\end{enumerate}
\end{construction}

See Figures \ref{fig:T8} and \ref{fig:T11} for $T_8^{(k)}$ and $T_{13}^{(k)}$, respectively.

\begin{figure}[ht]
\begin{center}
\begin{tikzpicture}[line cap=round,line join=round,>=triangle 45,x=.25cm,y=.25cm]
\draw(-2.,0) circle [x radius=.5625cm, y radius=.25cm, rotate=0];
\draw((2.,0) circle [x radius=.5625cm, y radius=.25cm, rotate=0];
\draw(5.5,0) circle [x radius=.5625cm, y radius=.25cm, rotate=0];
\draw(-5.5,1.323) circle [x radius=.5625cm, y radius=.25cm, rotate=-41.4];
\draw(-8.5,3.969) circle [x radius=.5625cm, y radius=.25cm, rotate=-41.4];
\draw(-11.5,6.615) circle [x radius=.5625cm, y radius=.25cm, rotate=-41.4];
\draw(-5.5,-1.323) circle [x radius=.5625cm, y radius=.25cm, rotate=41.4];
\draw(-8.5,-3.969) circle [x radius=.5625cm, y radius=.25cm, rotate=41.4];
\draw(-11.5,-6.615) circle [x radius=.5625cm, y radius=.25cm, rotate=41.4];
\end{tikzpicture}
\caption{A copy of $T^{(k)}_8$, as described in Construction \ref{con:k=k}.}\label{fig:T8}
\end{center}
\end{figure}

\begin{figure}[ht]
\begin{center}
\begin{tikzpicture}[line cap=round,line join=round,>=triangle 45,x=.25cm,y=.25cm]
\draw(-1.5,0.) circle [x radius=.45cm, y radius=.125cm, rotate=0];
\draw(-4.1,0.85) circle [x radius=.45cm, y radius=.125cm, rotate=-32];
\draw(-6.5,2.36) circle [x radius=.45cm, y radius=.125cm, rotate=-32];
\draw(-9.0,2.6) circle [x radius=.45cm, y radius=.125cm, rotate=15.7];
\draw(-11.7,1.9) circle [x radius=.45cm, y radius=.125cm, rotate=15.7];
\draw(-14.3,1.1) circle [x radius=.45cm, y radius=.125cm, rotate=15.7];
\draw(-8.7,4.06) circle [x radius=.45cm, y radius=.125cm, rotate=-43.86];
\draw(-10.66,5.96) circle [x radius=.45cm, y radius=.125cm, rotate=-43.86];
\draw(-12.6,7.87) circle [x radius=.45cm, y radius=.125cm, rotate=-43.86];

\draw(-4.1,-0.85) circle [x radius=.45cm, y radius=.125cm, rotate=32];
\draw(-6.5,-2.36) circle [x radius=.45cm, y radius=.125cm, rotate=32];
\draw(-9.0,-2.6) circle [x radius=.45cm, y radius=.125cm, rotate=-15.7];
\draw(-11.7,-1.9) circle [x radius=.45cm, y radius=.125cm, rotate=-15.7];
\draw(-14.3,-1.1) circle [x radius=.45cm, y radius=.125cm, rotate=-15.7];
\draw(-8.7,-4.06) circle [x radius=.45cm, y radius=.125cm, rotate=43.86];
\draw(-10.66,-5.96) circle [x radius=.45cm, y radius=.125cm, rotate=43.86];
\draw(-12.6,-7.87) circle [x radius=.45cm, y radius=.125cm, rotate=43.86];

\draw(1.5,0.) circle [x radius=.45cm, y radius=.125cm, rotate=0];
\draw(4.1,0.85) circle [x radius=.45cm, y radius=.125cm, rotate=32];
\draw(6.5,2.36) circle [x radius=.45cm, y radius=.125cm, rotate=32];
\draw(9.0,2.6) circle [x radius=.45cm, y radius=.125cm, rotate=-15.7];
\draw(11.7,1.9) circle [x radius=.45cm, y radius=.125cm, rotate=-15.7];
\draw(8.7,4.06) circle [x radius=.45cm, y radius=.125cm, rotate=43.86];
\draw(10.66,5.96) circle [x radius=.45cm, y radius=.125cm, rotate=43.86];

\draw(4.1,-0.85) circle [x radius=.45cm, y radius=.125cm, rotate=-32];
\draw(6.5,-2.36) circle [x radius=.45cm, y radius=.125cm, rotate=-32];
\draw(9.0,-2.6) circle [x radius=.45cm, y radius=.125cm, rotate=15.7];
\draw(11.7,-1.9) circle [x radius=.45cm, y radius=.125cm, rotate=15.7];
\draw(8.7,-4.06) circle [x radius=.45cm, y radius=.125cm, rotate=-43.86];
\draw(10.66,-5.96) circle [x radius=.45cm, y radius=.125cm, rotate=-43.86];

\end{tikzpicture}
\caption{A copy of $T^{(k)}_{13}$, as described in Construction \ref{con:k=k}.}
\label{fig:T11}
\end{center}
\end{figure}

\begin{observation}\label{k>=6 tree construction edge count observation}
If $k\geq 6$ and $m\geq 8$, then 
\[
\left|E\left(T^{(k)}_m\right)\right|=
\begin{cases} 
2^{s+1}+2^s+2^{s-1}+2^{s-2}-6&\text{ if }m=4s,\\
2^{s+2}+2^{s-1}-6&\text{ if }m=4s+1,\\
2^{s+2}+2^{s}-6&\text{ if }m=4s+2,\\
2^{s+2}+2^{s+1}+2^{s-1}-6&\text{ if }m=4s+3.
\end{cases}
\]
\end{observation}

\subsection{The saturation of $\mathbf{T^{(k)}_m}$}
We now provide a proof that $T^{(k)}_m$ is Berge-$P_m$-saturated for all $k\geq 6$ and $m\geq 10$. As the saturation of $T^{(3)}_m$ and $T^{(4)}_m$ follow from almost the same arguments that we will make for the case $k\geq 6$, we provide a sketch of how to adapt the $k\geq 6$ proof for these cases.

A vertex of degree $3$ in $T^{(k)}_m$ will be called a \emph{branch vertex} and an edge that contains three vertices of degree $2$ will be called a \emph{branch edge}. Note that $T^{(3)}_m$ and $T^{(4)}_m$ contain branch edges but not branch vertices, while $T^{(k)}_m$ with $k\geq 6$ contains branch vertices but not branch edges. With these definition in hand, we now prove the following:

	\begin{lemma}\label{tree saturation lemma}
	For all $k\geq 3$, $k\neq 5$ and $m\geq 10$, $T^{(k)}_m$ is Berge-$P_m$-saturated.
	\end{lemma}

\begin{proof}
	It is clear from the construction of $T^{(k)}_m$ that the longest path stretches from a leaf on the left to a leaf on the right, or in the case $k\geq 6$ and $4|m$, a leaf in one of the three main branches to a leaf in another main branch. This path is of length $m-2$, so $T^{(k)}_m$ is Berge-$P_m$-free. Now we show that any edge added will create a Berge-$P_m$.
	
	We first consider the case when $k\geq 6$ and $4\nmid m$. Let $e$ be any edge in $\overline{T^{(k)}_m}$. For each vertex $v\in V(T^{(k)}_m)$, let $\ell(v)$ and $r(v)$ be the length of the shortest path that starts at $v$ and ends on a leaf on the left, or right side of $T^{(k)}_m$ respectively. Note that for each $v$ $\ell(v)+r(v)=m-2$ if $d(v)\geq 2$ and $\ell(v)+r(v)=m-1$ if $d(v)=1$ since the edge containing $v$ will be used both on the left path and the right path.
	
	Let $x,y\in e$ with $\min\{\ell(x),r(x)\}\leq \min\{\ell(y),r(y)\}$. Let $P$ be a Berge-$P_{m-1}$ containing $x$. If $y$ is not contained in an edge of $P$, then we can traverse the longer part of $P$, hitting at least $m-2-\min\{\ell(x),r(x)\}$ edges until we hit $x$, then use $e$ to jump to $y$, and then take a path from $y$ on $\min\{\ell(y),r(y)\}$ edges to a leaf that does not contain $x$. This path has length at least
	\[
	m-2-\min\{\ell(x),r(x)\}+1+\min\{\ell(y),r(y)\}\geq m-1.
	\]
	Thus this path contains a desired Berge-$P_m$ in $T^{(k)}_m+e$.
	
	Then we may assume that each pair $x,y$ in $e$ with $\min\{\ell(x),r(x)\}\leq \min\{\ell(y),r(y)\}$ has that $y$ is in every longest path containing $x$. This implies that all the vertices in $e$ are in a single longest path $P^*$.  Further, this implies that every vertex is on one side of the construction or in the center since if we had vertices properly on both sides of the construction, there would be a Berge-$P_{m-1}$ that contains a vertex on one side, but not the vertex on the other side due to the branching structure of $T^{(k)}_m$. Say that the vertices are all on the left side or in the center. Let $x^*\in e$ be such that $\ell(x^*)=\min\{\ell(v)\mid v\in e\}$ and let $y^*$ be such that $\ell(y^*)=\max\{\ell(v)\mid v\in e\}$. 
	
	If $d(x^*,y^*)\geq 4$, then there is a branch vertex between $x^*$ and $y^*$ that is not adjacent to $x^*$. We can traverse $P^*$ until we get to $y^*$ picking up $r(y^*)$ edges, then hop to $x^*$, then traverse $P^*$ backwards until we hit the last branch vertex of $P^*$ before $y^*$ hitting at least $d(x^*,y^*)-2$ edges, then take this branch down to a leaf using at least $\ell(y^*)-2$ edges, giving us a path of length at least
	\[
	r(y^*)+1+d(x^*,y^*)-2+\ell(y^*)-2\geq r(y^*)+\ell(y^*)+1\geq m-1.
	\]
	Thus we can assume that  $d(x^*,y^*)\leq 3$.
	
	Note that this preceding path also works if $d(x^*,y^*)=3$ as long as either $d(y^*)=1$, since then $r(y^*)+\ell(y^*)= m-1$, or $d(y^*)=2$ since then we hit $d(x^*,y^*)-1$ edges backtracking along $P^*$ and $\ell(y^*)-1$ edges going down to a leaf since the branching point is only one edge away from $y^*$. Thus we may assume $d(x^*,y^*)\leq 3$ and if equality holds, $d(y^*)=3$.
	
	If $e$ contains a pair of vertices $u_1$ and $u_2$ such that $d(u_1)\geq 2$, $d(u_2)=1$ and $d(u_1,u_2)=1$, or such that $d(u_1)=d(u_2)=1$ and $d(u_1,u_2)=2$, then $e$ can be added into any path of length $m-2$ that contains the edge containing $u_1$ and $u_2$ by traversing the path until $u_1$, then using $e$ to jump from $u_1$ to $u_2$, then traversing the rest of the path. Thus we are done unless $d(x^*,y^*)\geq 3$ since there is no way for two adjacent edges to contain all $k$ vertices without having a pair of vertices that can play the roles of $u_1$ and $u_2$ above.
	
	Then we have that $d(x^*,y^*)=3$ and $d(y^*)=3$. Let $e_1'$, $e_2'$ and $e_3'$ be the edges in the path between $x^*$ and $y^*$ with $x^*\in e_1'$ and $y^*\in e_3'$. Now, $e\subset e_1'\cup e_2'\cup e_3'$. We can assume that there are no vertices from $e$ in $e_3'\setminus e_2'$ aside from $y$ since any such vertex along with $y$ would satisfy the conditions on $u_1$ and $u_2$ above.
	
	It could be the case that $e_1'$ is a leaf, as in the left diagram of Figure \ref{fig: saturation figure}. In this case, $d(x^*)=1$, so any vertex in $e_2'\setminus e_3'$ along with $x^*$ would satisfy the conditions on $u_1$ and $u_2$ above, and if any vertex in $e_1'\setminus e_2'$ aside from $x^*$ was is $e$, then we could extend a path of length $m-2$ that ends at $x^*$ by $1$ edge using $e$ to jump from $x^*$ to the other vertex in $e_1'\cap e$. Finally any degree $1$ vertex in $e_3'$ along with $y^*$ satisfy the conditions on $u_1$ and $u_2$ above, so none of these vertices can be in $e$. This only leaves the vertex in $e_2'\cap e_3'$, but $k\geq 6$ giving us a contradiction with the size of the edges $e$. Thus we can assume $e_1'$ is not a leaf, as in the right hand diagram of Figure \ref{fig: saturation figure}.
	
			\begin{figure}[ht]
			\begin{center}
				\begin{tikzpicture}[scale=0.7,line cap=round,line join=round,>=triangle 45,x=.5cm,y=.5cm]
				
				\draw [rotate around={45:(3,3)}] (3,3) ellipse (.9cm and 0.25cm);
				\draw [rotate around={45.:(5,5)}] (5,5) ellipse (.9cm and 0.25cm);
				\draw [rotate around={45:(7,7)}] (7,7) ellipse (.9cm and 0.25cm);

				\draw [rotate around={45:(1,1)}] (1,1) ellipse (.9cm and 0.25cm);
				\draw [rotate around={45.:(-1,-1)}] (-1,-1) ellipse (.9cm and 0.25cm);
				\draw [rotate around={45:(-3,-3)}] (-3,-3) ellipse (.9cm and 0.25cm);
				
				\draw [rotate around={-45.:(3,1)}] (3,1) ellipse (.9cm and 0.25cm);
				\draw [rotate around={-45.:(5,-1)}] (5,-1) ellipse (.9cm and 0.25cm);
				\draw [rotate around={-45.:(7,-3)}] (7,-3) ellipse (.9cm and 0.25cm);
				
				\draw [rotate around={-45.:(7,5)}] (7,5) ellipse (.9cm and 0.25cm);
				
				\draw[color=black] (9.4,7.7) node {$\dots$};
				\draw[color=black] (9.4,4.3) node {$\dots$};
				
				\draw[color=black] (-3.3,-3) node[draw,shape=circle,fill=black,scale=0.13,label=above:$x^*$] {$x^*$};
				\draw[color=black] (2,2) node[draw,shape=circle,fill=black,scale=0.13,label=above:$y^*$] {$y^*$};
				
				\end{tikzpicture}
				\qquad
				\begin{tikzpicture}[scale=0.7,line cap=round,line join=round,>=triangle 45,x=.5cm,y=.5cm]
				
				\draw [rotate around={45:(3,3)}] (3,3) ellipse (.9cm and 0.25cm);
				\draw [rotate around={45.:(5,5)}] (5,5) ellipse (.9cm and 0.25cm);

				\draw [rotate around={45:(1,1)}] (1,1) ellipse (.9cm and 0.25cm);
				\draw [rotate around={45.:(-1,-1)}] (-1,-1) ellipse (.9cm and 0.25cm);
				\draw [rotate around={45:(-3,-3)}] (-3,-3) ellipse (.9cm and 0.25cm);
				
				\draw [rotate around={-45.:(3,1)}] (3,1) ellipse (.9cm and 0.25cm);
				\draw [rotate around={-45.:(5,-1)}] (5,-1) ellipse (.9cm and 0.25cm);

				\draw [rotate around={-45.:(-1,-3)}] (-1,-3) ellipse (.9cm and 0.25cm);
				\draw [rotate around={45.:(-5,-5)}] (-5,-5) ellipse (.9cm and 0.25cm);
				\draw [rotate around={-45.:(1,-5)}] (1,-5) ellipse (.9cm and 0.25cm);

				\draw[color=black] (7.5,5.9) node {$\dots$};
				\draw[color=black] (7.5,-1.9) node {$\dots$};
				\draw[color=black] (-7.5,-5.9) node {$\dots$};
				\draw[color=black] (3.5,-5.9) node {$\dots$};
				
				\draw[color=black] (-3.3,-3) node[draw,shape=circle,fill=black,scale=0.13,label=above:$x^*$] {$x^*$};
				\draw[color=black] (2,2) node[draw,shape=circle,fill=black,scale=0.13,label=above:$y^*$] {$y^*$};
				
				\end{tikzpicture}
				\caption{The two cases when $d(x^*,y^*)=3$ and $d(y^*)=3$.}
				\label{fig: saturation figure}
			\end{center}
		\end{figure}
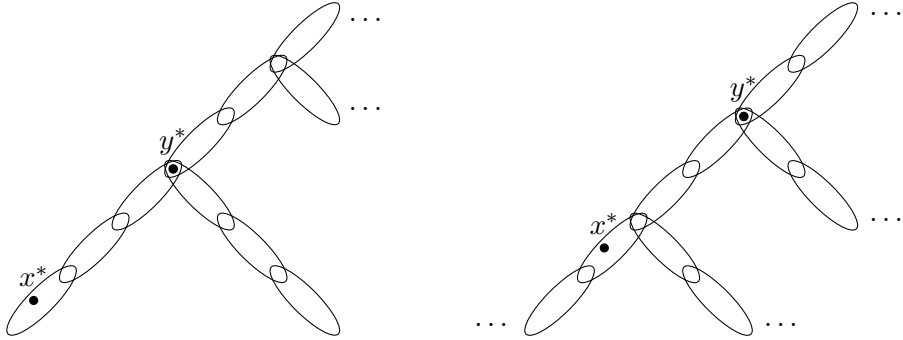
		
	If $d(x^*)\geq2$ in $T_m^{(k)}$, then by the preceding argument, there is at most one vertex in $e$ at distance $1$ from $x^*$, and if $d(x^*)=1$, then by the preceding argument, there are at most $k-3$ vertices from $e$ at distance $1$ from $x^*$ since they all must be degree $1$ vertices. Thus in all cases, there are at least two vertices not in the same edge as $x^*$. One is $y^*$, call the other $w^*$.
	
	Since $y^*$ has degree $3$ and $d(y^*,w^*)\leq 2$, there is a branch vertex between $w^*$ and $x^*$, and $d(w^*)\leq 2$ since this branch vertex is distance $1$ from $x^*$.
	
	Then we can traverse $P^*$ until we hit $w^*$ picking up $r(w^*)$ edges, then jump to $x^*$, pick up a single edge going to the branch vertex, then take a path down from the branching point using $\ell(w^*)-2$ edges if $d(w^*)=1$ and $d(w^*,y^*)=1$ or $\ell(w^*)-1$ otherwise. If $d(w^*)=1$ and $d(w^*,y^*)=1$, then this gives us at path on 
	\[
	r(w^*)+1+1+\ell(w^*)-2=r(w^*)+\ell(w^*)=m-1
	\]
	edges and otherwise, a path on
	\[
	r(w^*)+1+1+\ell(w^*)-1=r(w^*)+\ell(w^*)+1\geq m-1
	\]		
	edges, a desired path in either case. Thus $T^{(k)}_m$ is saturated when $4\nmid k$.
	
	When $k\geq 6$ and  $4\mid k$, the proof follows from a similar argument. The only difference is that instead of having two sides, a left side and a right side, $T^{(k)}_m$ has three branches coming from the center. Thus instead of having $\ell(v)$ and $r(v)$, we can label the three branches with integers $1,2$ and $3$, and define $\ell_i(v)$ for $1\leq i\leq 3$ by the distance from $v$ to a leaf in branch $i$. This will give us that all the vertices of $e$ are on one path inside one main branch. From here, the proof is identical.
	
	For $k=3$ and $k=4$, again the proof follows similarly. We can argue that all the vertices of $e$ must lie on one path and all on one side of the tree in exactly the same ways as $k\geq 6$. Define $x^*$ and $y^*$ in the same way as in the proof of $k\geq6$. For $k=3$, we can establish that $d(x^*,y^*)\leq 3$ the same as before, but with branch edges in place of branch vertices. In the case of $k=4$, we cannot prove that $d(x^*,y^*)\leq 3$ as easily since the branch edges in $T^{(4)}_m$ are further away from eachother than branch edges in  in $T^{(3)}_m$. Even so, using essentially the same idea we can establish $d(x^*,y^*)\leq 5$. From here, for both $k=3$ and $k=4$, there are only a few small cases to work out similar to the end of the proof for $k\geq 6$.
\end{proof}

It is worth noting that in Lemma \ref{tree saturation lemma} we assume $m\geq 10$ even though we define some constructions with $m<10$, such as $T^{(3)}_8$ or $T^{(172)}_9$. These constructions are also saturated; we choose $m\geq 10$ to simplify the statement of Lemma \ref{tree saturation lemma} since we do not provide constructions for $k=4$ and $m<10$. 

\subsection{Using $\mathbf{T^{(k)}_m}$ to construct $\mathbf{H^{(k)}_{n,m}}$}

We will use the tree $T^{(k)}_m$ as the building block to construct a Berge-$P_m$ saturated hypergraph on $n$ vertices which we will call $H^{(k)}_{n,m}$.

\begin{construction}\label{H construction}
Let $k\geq 3$, $k\neq 5$ and $m\geq 10$. 

Let $r=n \mod ((k{-}1)|E(T^{(k)}_m)|{+}1$ and let $H$ be the hypergraph on $n$ vertices that contains $\left\lfloor{\frac{n}{(k{-}1)|E(T^{(k)}_m)|{+}1}}\right\rfloor$ components isomorphic to $T^{(k)}_m$ and $r$ isolated vertices, $v_1,\dots,v_r$. If $r>0$, we will show how to incorporate these vertices into one of the copies of $T^{(k)}_m$, which we will call $T'$.
 
 First, we will create new leaves in $T'$ using the isolated vertices in $H$. Let $v'\in V(T')$ be the vertex of degree $2$ in some leaf of $T'$ and for each $0\leq i\leq\left\lfloor \frac{r}{k{-}1}\right\rfloor-1$, let $e_i=\{v', v_{(k-1)i+1},v_{(k-1)i+2},...,v_{(k-1)i+k}\}$. Then let $H'$ be a hypergraph with $V(H')=V(H)$ and $E(H')=E(H)\cup\bigcup_{i=0}^{\left\lfloor \frac{r}{k{-}1}\right\rfloor-1}e_i$.
 
 Then $H'$ has $r'=r \mod (k{-}1)$ isolated vertices. If $r'>0$, the method for including the remaining $r'$ vertices is dependent upon the uniformity, but for each $k$ we will add one additional edge, $e'$, which will be incident to the same edges as some specific edge $e^*\in T^{(k)}_m$.

For $k=3$, we must have $r'=1$. This vertex can be added by ``cloning'' the central edge of $T'$ meaning we form an edge using the remaining new vertex and the two vertices of degree $2$ in the center edge of $T'$. In this case $e^*$ is the central edge.
For $k=4$, consider the left three initial edges, $e_1,e_2$, and $e_3$, labeling from the left. 
Let $e'$ be the edge containing the $r'$ isolated vertices, the vertex of degree 2 in $e_1\cap e_2$, and $3-r'\geq 1$ vertices of degree $1$ in $e_3$. In this case $e^*=e_2$.
For $k\geq 6$, let $e'$ be the edge containing the $r'$ isolated vertices, the degree $3$ vertex in one of the center edges, and $k{-}(r'{+}1)$ vertices of degree $1$ from the other center edge as in Figure \ref{fig: absorbingVertices}. 

Then we let $H^{(k)}_{n,m}=H'+e'$.
\end{construction}

	\begin{figure}[ht]
		\begin{center}
			\begin{tikzpicture}[line cap=round,line join=round,>=triangle 45,x=.5cm,y=.5cm]
			\clip(-9,-3.5) rectangle (9,3.5);
			\draw [rotate around={0.:(-1.5,0.)}] (-1.5,0.) ellipse (.9cm and 0.25cm);
			\draw [rotate around={-32.:(-4.1,0.85)}] (-4.1,0.85) ellipse (.9cm and 0.25cm);
			\draw [rotate around={-32:(-6.5,2.36)}] (-6.5,2.36) ellipse (.9cm and 0.25cm);
			
			\draw [rotate around={32.:(-4.1,-0.85)}] (-4.1,-0.85) ellipse (.9cm and 0.25cm);
			\draw [rotate around={32:(-6.5,-2.36)}] (-6.5,-2.36) ellipse (.9cm and 0.25cm);
			
			\draw [rotate around={0.:(1.5,0.)}] (1.5,0.) ellipse (.9cm and 0.25cm);
			\draw [rotate around={32.:(4.1,0.85)}] (4.1,0.85) ellipse (.9cm and 0.25cm);
			\draw [rotate around={32:(6.5,2.36)}] (6.5,2.36) ellipse (.9cm and 0.25cm);

			\draw [rotate around={-32.:(4.1,-0.85)}] (4.1,-0.85) ellipse (.9cm and 0.25cm);
			\draw [rotate around={-32:(6.5,-2.36)}] (6.5,-2.36) ellipse (.9cm and 0.25cm);
			
			\draw[dashed] (-.7,0) .. controls (-.2,-1) and (.2,-1) .. (2.5,0); 
			\draw[dashed] (2.5,0) .. controls (2.7,.2) and (2.9,.2) .. (3.2,0); 
			\draw[dashed] (-1.6,0) .. controls (-.6,-1.5) and (.6,-2) .. (3.2,0); 
			\draw[dashed] (-1.6,0) .. controls (-1.3,.2) and (-1,.2) .. (-.7,0); 
			
			\draw[dotted, line width=.3mm](-8.8,3.3)--(-8.3,3.3);
			\draw[dotted, line width=.3mm](-8.8,-3.3)--(-8.3,-3.3);
			\draw[dotted, line width=.3mm](8.8,3.3)--(8.3,3.3);
			\draw[dotted, line width=.3mm](8.8,-3.3)--(8.3,-3.3);
			
			\draw[color=black] (1.4,-1.5) node {$e'$};
			\draw[color=black] (1.5,0.85) node {$e^*$};
			
			\end{tikzpicture}
			\caption{A copy of $T'$, for $k\geq6$, as described in Construction \ref{H construction}.}
			\label{fig: absorbingVertices}
		\end{center}
	\end{figure}

\begin{observation}\label{H edge count observation}
	Let $k\geq 3$ with $k\neq 5$, and let $m\geq 10$. Then for all $n\geq (k-1)\left|E\left(T^{(k)}_m\right)\right|+1$, 
\[
\left|E\left(H^{(k)}_{n,m}\right)\right|=\left\lceil\frac{1}{k-1}\left(n-\left\lfloor{\frac{n}{(k-1)|E(T^{(k)}_m)|+1}}\right\rfloor\right)\right\rceil
\]
\end{observation}

The preceding observation follows from the fact that $H^{(k)}_{n,m}$ has $\left\lfloor{\frac{n}{(k-1)\left|E\left(T^{(k)}_m\right)\right|+1}}\right\rfloor$ components, at most one of which is not a linear tree, and if $H^{(k)}_{n,m}$ does contain a component that is not a linear tree, this component is minimal in the sense that there are no connected $k$-uniform hypegraphs on the same vertex set with fewer edges.

\begin{lemma}\label{H saturation lemma}
	$H^{(k)}_{n,m}$ is saturated for all $k\geq 3$ with $k\neq 5$ and all $m\geq 10$.
\end{lemma}

\begin{proof}
 Fix $k$ and $m$. Let $H=H^{(k)}_{n,m}$. $H$ consists of several components isomorphic to $T^{(k)}_m$, and one component $T'$, which is a copy of $T^{(k)}_m$ that may have been modified by adding copies of a particular leaf at the vertex $v'$ and possibly adding an edge near the center called $e'$. 
 
 We first show $H$ contains no Berge-$P_m$. Indeed, from Lemma \ref{tree saturation lemma} we know that no component isomorphic to $T^{(k)}_m$ contains a Berge-$P_m$, so we can restrict our attention to $T'$. Since there already was a leaf at $v'$, any added leaves could not create a Berge-$P_m$. If $e'\in E(H)$, consider a longest path, $P$ in $T'$ using $e'$.  Let $e^*$ be as described in Construction \ref{H construction}. If $P$ does not use $e^*$, then we can replace $e'$ with $e^*$ to create a path that only uses edges in $T^{(k)}_m$, with the same length, and thus is not a Berge-$P_m$. On the other hand if $P$ uses both $e^*$ and $e'$, for $k=3$, both vertices of degree 3 in $e'$ will be used by $P$, and so one end of $P$ must be in $e^*$ or $e'$. Similarly, for $k=4$ or $k\ge6$, $P$ must use the vertex of degree 3 in $e^*\cap e'$ and the other center edge that is incident to both $e^*$ and $e'$, and so one end of $P$ must be $e^*$ or $e'$. Thus, if $P$ uses both $e^*$ and $e'$, $P$ has length at most $\FL{\frac{m}{2}}+2$, which for $m\ge 10$ is less than $m-1$. Thus $H$ contains no Berge-$P_m$.

Now we will show that for all $e\in E(\overline{H})$, $H+e$ contains a Berge-$P_m$. By Lemma \ref{tree saturation lemma}, each $T^{(k)}_m$ is saturated, and so any edge added that is contained completely in one of the $T^{(k)}_m$ results in a Berge-$P_m$. It remains to show that $T'$ is saturated, and that any edge added between components results in a Berge-$P_m$. 

Consider adding an edge $e\in \overline{T'}$ to $T'$. If $e$ contains only vertices in the underlying $T^{(k)}_m$, then $T'+e$ contains a Berge-$P_m$. If $e$ contains vertices from more than one leaf incident with $v'$, any Berge-$P_{m-1}$ that has one of those leaves as a terminal edge can be extended in length by $1$, thus we can assume $e$ is incident with at most one leaf incident with $v'$, and without loss of generality, we can assume that leaf is the one in $T^{(k)}_m$. Thus, the only case we need still check is when $e$ contains a vertex in $e'$ that is not in the underlying $T^{(k)}_m$.
	
	If $e$ contains distinct vertices in $e^*$ and $e'$, then any Berge-$P_{m-1}$ that contains $e^*$ in the underlying $T^{(k)}_m$ can be extended to a Berge-$P_m$ containing all the previous edges and both $e'$ and $e^*$ since $e'$ is incident to the same edges as $e$. Otherwise, $e'$ acts like $e^*$, so any edge incident with vertices from $e'$ will create a Berge-$P_m$ just as if the edge was incident with vertices from $e^*$. Thus, $T'$ is saturated.
	
Finally, we show that any edge added between components results in a Berge-$P_m$. 
If we add an edge between any two components of $H$, then since every vertex of $T^{(k)}_m$, or $T'$ is in some Berge-$P_{m-1}$, each vertex is the initial vertex of a path of length at least $\CL{\frac{m-2}2}$, and so we can construct a path of length at least $2\CL{\frac{m-2}2}+1\geq m-1$. 
Thus, in all cases, the addition of an edge creates a Berge-$P_m$, so $H$ is Berge-$P_m$-saturated.
\end{proof}

We have constructed Berge-$P_m$ saturated hypergraphs. In Section \ref{section lower bound}, we will show that these hypergraphs are  edge minimal.

\section{Lower Bound on $\mathbf{a^{(k)}_m}$}\label{section lower bound}

This section will establish a lower bound on $a^{(k)}_m$, the fewest number of edges in any $k$-uniform Berge-$P_m$-saturated linear tree on at least $k+1$ vertices for $k=3$, $k=4$ and $k\geq 6$. The structure of such trees is highly dependent on the value of $k$, so this section provides proofs for each of three cases: $k=3$, $k=4$, and $k\geq 6$. These proofs are very similar to each other, but have non-trivial differences.
	
	All cases rely on Lemmas~\ref{twoedge} and \ref{edge in m-3}. Furthermore, in all cases we need to rule out the possibility that a minimal Berge-$P_m$-saturated tree does not contain a Berge-$P_{m-1}$. This is done in Lemmas \ref{no P_{m-1} branching structure}, \ref{no P_{m-1} counting lemma} and \ref{no P_{m-1} final count lemma}. Aside from these lemmas, the proof for each case is independent of the others. If the reader does not wish to read all three cases, the authors recommend one of the following two reading paths in addition to the five lemmas above:
	
	The simplest case to understand is $k=3$. This case follows from Lemmas \ref{k=3 P_{m-1} branching lemma} and \ref{k=3 P_{m-1} counting lemma} and concludes in Theorem~\ref{k=3 P_{m-1} final count theorem}. 
	
	If the reader instead wishes to read the most general case of $k\geq 6$, this consists of Lemmas \ref{k>= 6 P_{m-1} branching lemma} and \ref{k>= 6 P_{m-1} counting lemma} and concludes in Theorem~\ref{k>=6 P_{m-1} final count theorem}. The reader may also wish to read each of the paragraphs at the beginning of Sections \ref{section branching lemmas}, \ref{section counting lemmas} and \ref{section central structure} as these paragraphs foreshadow what follows in each subsection.
	
	Now we can begin to establish the lower bound. Before we prove the lower bound on $a^{(k)}_m$, we need to develop a few elementary structural lemmas and make a few observations about Berge-$P_m$-saturated graphs.

\begin{lemma}\label{twoedge}
	Let $k\geq 3$. If $H$ is a $k$-uniform Berge-$P_m$-saturated linear tree on $n\geq k+1$ vertices, and $e_1$ and $e_2$ are a pair of adjacent edges in $H$, then there is a pair of vertex-disjoint paths of length $\alpha$ and $\beta$ that start at vertices in $e_1\cup e_2$ and do not use either edge $e_1$ or $e_2$ such that $\alpha+\beta\geq m-4$.
\end{lemma}

\begin{proof}
	Let $e_1$ and $e_2$ be a pair of adjacent edges. Let $e$ be any edge in $\overline{H}$ such that each vertex of $e$ is in either $e_1$ or $e_2$. Consider $H+e$. Due to saturation, there exists a Berge-$P_m$ in $H+e$ that uses the edge $e$. In addition to $e$, this Berge-$P_m$ could possibly use edges $e_1$, $e_2$ and edges from two edge-disjoint paths leaving $e_1$ or $e_2$, but no more edges. Thus if the longest such paths are of length $\alpha$ and $\beta$, we must have that $\alpha+\beta+3\geq m-1$ or $\alpha+\beta\geq m-4$. 
	
	To see that these two paths are vertex-disjoint, assume that the only pair of paths long enough to satisfy the length requirement attach at the same vertex $x\in e_1\cup e_2$. Consider an $e'\subseteq (e_1\cup e_2)\setminus\{x\}$. Then, $H+e'$ will not contain a Berge-$P_m$ since we cannot use $e'$ while also using both paths starting at $x$, since this would require us to visit $x$ twice in our path.
\end{proof}

\begin{observation}\label{Pm-2}
	Let $k\geq 3$. If $H$ is a $k$-uniform Berge-$P_m$-saturated linear tree on $n\geq k+1$ vertices, then $H$ contains a Berge-$P_{m-2}$ (a path of length $m-3$).
\end{observation}
This follows immediately from applying Lemma~\ref{twoedge} to any pair of edges in $H$.

\begin{lemma}\label{edge in m-3}
	Let $k\geq 3$. If $H$ is a $k$-uniform Berge-$P_m$-saturated linear tree on $n\geq k+1$ vertices, then every edge in $H$ that is not a leaf is in a Berge-$P_{m-2}$. 
\end{lemma}

\begin{proof}
Let $e$ be an edge of $H$ that is not a leaf. Since $e$ is not a leaf, it must be in a path of length at least $3$, say with neighboring edges $e'$ and $e''$.
Apply Lemma~\ref{twoedge} to edge $e$ and $e'$. Then we have paths of length $\alpha_1$ and $\beta_1$ leaving distinct vertices in $e\cup e'$ with $\alpha_1+\beta_1\geq m-4$. If either of these paths attach at a vertex in $e\setminus e'$, then we are done since regardless of where the second path attaches, we can create a path using $e$ and both the path of length $\alpha_1$ and $\beta_1$, which is of length at least $m-3$.

Otherwise both paths attach to vertices in $e'$.
We now apply Lemma~\ref{twoedge} to $e$ and $e''$ to get paths of length $\alpha_2$ and $\beta_2$ leaving  distinct vertices in $e\cup e''$
If either one leaves a vertex in $e\setminus e''$, then we are done by the argument in the preceding paragraph, so assume both attach to vertices in $e''$.
Without loss of generality let $\alpha_1\geq \beta_1$ and $\alpha_2\geq \beta_2$ be the longer paths generated by the two applications of Lemma~\ref{twoedge}.
We note that $\alpha_i\ge\CL{\frac{m-4}{2}}$ for $i\in\{1,2\}$.
Thus, there is a path through $e$ that traverses each of these longer paths, so it has length at least $\CL{\frac{m-4}{2}}+\CL{\frac{m-4}{2}}+1\ge m-3$.
\end{proof}

\subsection{Branching Lemmas}\label{section branching lemmas}

If we consider a loose path on $m-2$ edges, it is clear this will not be Berge-$P_m$-saturated since we can add edges that contain only vertices of degree $2$ or edges that contain vertices of degree $1$ from two non-consecutive non-leaf edges. Due to these restrictions, any path in a Berge-$P_m$-saturated graph must branch fairly often. This section establishes lemmas that characterize the minimal necessary branching of a Berge-$P_m$-saturated tree. We will find that the minimal amount of branching depends on if the path we are currently considering contains a Berge-$P_{m-1}$ or not.

\subsubsection{Branching with a Berge-$\mathbf{P_{m-1}}$}

If $H$ contains a Berge-$P_{m-1}$, the amount of branching necessary to retain saturation is dictated by the uniformity $k$. Our first lemma is true for all $k\geq 3$, but is trivial for $k=3$ and can be refined for $k=4$. Also in this section is the refinement for $k=4$ and a non-trivial branching lemma for $k=3$.

\begin{lemma}\label{k>= 6 P_{m-1} branching lemma}
	Let $k\geq 3$. Let $H$ be a $k$-uniform Berge-$P_m$-saturated linear tree on $n\geq k+1$ vertices. Let $e_1$, $e_2$ and $e_3$ be sequential edges in some Berge-$P_{m-1}$ $P$. Let there be $\alpha$ edges in $P$ preceding $e_1$ attached at a vertex $x\in e_1$, and $\beta$ edges following $e_3$ with $\alpha\geq \beta\geq 0$. Let $y\in e_1\cap e_2$ and $z\in e_2\cap e_3$. Then one of the following is satisfied:
	
	(a) There is a path of length at least $\beta$ which is edge-disjoint from $P$ that attaches to a vertex in $e_2$
	
	or
	
	(b) there are paths of length at least $\beta$ edge-disjoint from $P$ attaching to at least $k-1$ vertices in $X=(e_1\cup e_3)\setminus e_2$.
\end{lemma}

\begin{proof}
	For $k=3$, $|X|=|(e_1\cup e_3)\setminus e_2|=4$ and the subpaths of $P$ of length $\alpha$ and $\beta$ attach to $k-1=2$ vertices in $X$, so (b) holds.
	
	Now consider, $k\geq 4$. Observe that due to the length of $P$, $\alpha+\beta+3=m-2$. Toward a contradiction, assume that there is no path of length at least $\beta$ leaving $e_2$ that is edge disjoint from $P$, and there are at least $k$ vertices in $(e_1\cup e_3)\setminus \{y,z\}$ that do not have paths of length $\beta$ attached to them. Consider the edge $e$ made up of these $k$ vertices as shown in Figure~\ref{fig:H+e k>=6}. Since $H$ is saturated, $H+e$ contains a Berge-$P_m$ that uses $e$.

	\begin{figure}  [h]
		\begin{center}  
			\begin{tikzpicture}[line cap=round,line join=round,>=triangle 45,x=1.0cm,y=1.0cm]
			\clip(-2,-2) rectangle (10,1);
			\draw [rotate around={0.:(1.3,0.045)}] (1.30,0.045) ellipse (1.55cm and 0.39cm); 
			\draw [rotate around={0.:(3.88,0)}] (3.88,0.05) ellipse (1.55cm and 0.39cm); 
			\draw [rotate around={0.:(6.446,0.025)}] (6.446,0.025) ellipse (1.55cm and 0.39cm); 
			\begin{scriptsize}
			\draw[color=black] (1.2,.6) node {$e_1$};
			\draw[color=black] (3.73,.6) node {$e_2$};
			\draw[color=black] (6.28,.6) node {$e_3$};
			\draw [fill=black] (2.59,0) circle (1.5pt);
			\draw[color=black] (2.59,-0.376) node {$y$};
			\draw [fill=black] (5.16,0.0) circle (1.5pt);
			\draw[color=black] (5.16,-0.4246) node {$z$};
			\draw [fill=black] (0.048,0.0) circle (1.5pt);
			\draw[color=black] (0.226,0.0) node {$x$};
			\draw [fill=black] (7.7,0.0) circle (1.5pt);
			
			\draw (0.048,0) parabola (-2,-2);\draw[color=black] (-1,-.8) node {$\alpha$};
			\draw (7.7,0)parabola  (10,2);\draw[color=black] (8.7,.6) node {$\beta$};
			\draw (.7,0) .. controls (2.59,-2) and (5.16,-2) .. (7.2,-.1); 
			\draw (7.2,-.1) .. controls (6.8,.3) and (6.3,.3) .. (5.9,.1);
			\draw (5.9,.1) .. controls (5,-1.2) and (3,-1.5) .. (2,.1);
			\draw (2,.1) .. controls (1.4,.3) and (1,.3) .. (.7,0);
			\draw[color=black] (5,-1.5) node {$e$};
			\end{scriptsize}
			\end{tikzpicture}
			\end{center}   
			\caption{\label{fig:H+e k>=6}The graph $H+e$ as described in Lemma~\ref{k>= 6 P_{m-1} branching lemma}.}
			\end{figure}
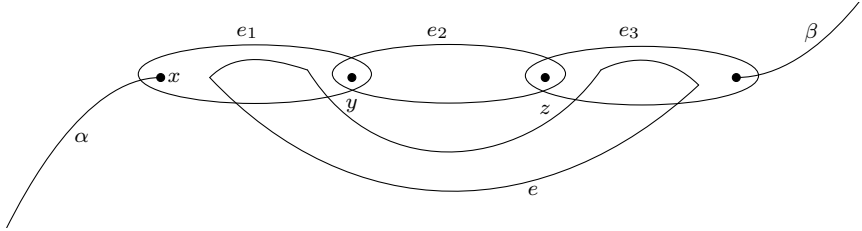
			
			This Berge-$P_m$ can use at most all four of the edges $e_1, e_2, e_3$ and $e$ and two paths that each start at some vertex in $e_1\cup e_2\cup e_3$. By the maximality of $P$, these two paths are of length at most $\alpha$, and if one was of a length $\gamma<\beta$, then the longest path using $e$ is of length at most $\alpha+\gamma+4<\alpha+\beta+4=m-1$, a contradiction, so they are of length at least $\beta$. Thus, the Berge-$P_m$ cannot have used some path that entered at $e_2$ since there are no paths of length $\beta$ attached to a vertex in $e_2$. If the Berge-$P_m$ used two paths that both entered at $e_1$ or both entered at $e_3$, then the edge, $e_1$ or $e_3$ respectively would have to be used twice since the vertices in $e$ are not attached to paths of length $\beta$.
			
			Then the remaining possibility is that the Berge-$P_m$ uses one path entering at $e_1$, and a second path leaving at $e_3$. By the maximality of $P$, the first path is of length at most $\alpha$ and the second path at most $\beta$. Note that there is no way to traverse these paths while also using both $e$ and $e_2$, so this path is of length at most $\alpha+\beta+3<m-1$, a contradiction. In all cases we reach a contradiction, so either (a) or (b) occurs.
			\end{proof}

\begin{corollary}\label{k=4 P_{m-1} branching corollary}
	Let $H$ be a $4$-uniform Berge-$P_m$-saturated linear tree on $n\geq 10$ vertices. Let $e_1$, $e_2$ and $e_3$ be sequential edges in some Berge-$P_{m-1}$, say $P$. Let there be $\alpha$ edges in $P$ preceding $e_1$ attached at a vertex $x$, and $\beta$ edges following $e_3$ attached at a vertex $y$ with $\alpha\geq \beta\geq 0$. Then there is a path of length at least $\beta$, edge disjoint from $P$ that starts at a vertex in $(e_1\cup e_2\cup e_3)\setminus \{x,y\}$.
\end{corollary}

\begin{proof}
	Apply Lemma~\ref{k>= 6 P_{m-1} branching lemma} to edges $e_1$, $e_2$ and $e_3$ in path $P$. If (a) is satisfied, we are done. If (b) is satisfied, then we have paths of length $\beta$ leaving at least $k-1=3$ vertices in $(e_1\cup e_3)\setminus e_2$. The subpaths of $P$ of length $\alpha$ and $\beta$ can count for two of these, but there must be a third path on at least $\beta$ edges, edge-disjoint from $P$, attaching to a vertex other than $x$ or $y$, so in all cases we are done.
\end{proof}

\begin{lemma}\label{k=3 P_{m-1} branching lemma}
	Let $H$ be a $3$-uniform Berge-$P_m$ saturated linear tree that contains a Berge-$P_{m-1}$ on $n\geq 4$ vertices. Let $e_1=x_1x_2x_3$ and $e_2=x_3x_4x_5$ be sequential edges in some Berge-$P_{m-1}$, say $P$, with $\alpha$ edges preceding $e_1$, attaching at vertex $x_1$ and $\beta$ edges after $e_2$, attaching at $x_5$ with $\alpha\geq \beta$. Then there exists a path of length at least $\beta$ edge disjoint from $P$ that starts at a vertex in $\{x_2,x_3,x_4\}$.
\end{lemma}

\begin{proof}
	By the length of $P$, we have that $\alpha+\beta+2=m-2$. Let $e=\{x_1,x_3,x_5\}$. Then by saturation $H+e$ contains a Berge-$P_m$ that uses the edge $e$. Observe that any such Berge-$P_m$ cannot traverse a path ending at $x_1$, then use all three edges $e_1$ $e_2$ and $e$, then traverse a path leaving $x_5$ by the choice of $e$, so it must use a path starting at a vertex in $\{x_2,x_3,x_4\}$, say of length $\gamma$. Then this path can traverse a path of length at most $\alpha$, the three edges $e_1$, $e_2$ and $e$, then the path of length $\gamma$, so $\alpha+\gamma+3\geq m-1=\alpha+\beta+3$, so $\gamma\geq \beta$ as desired.
\end{proof}

\subsubsection{Branching with no Berge-$\mathbf{P_{m-1}}$}

If $H$ does not contain a Berge-$P_{m-1}$, then we get the same minimal amount of branching for all uniformities.

\begin{lemma}\label{no P_{m-1} branching structure}
	Let $H=(V,E)$ be a $k$-uniform Berge-$P_m$-saturated linear tree with $k\geq 3$, and let $e_1$ be an edge that is not in any Berge-$P_{m-1}$. Let $P$ be a longest path in $H$ that contains $e_1$. Assume there are $\alpha$ edges in $P$ that precede $e_1$ and attach via $x\in e_1$ and assume there are $\beta$ edges that follow $e_1$, and attach via $y\in e_1$ with $\alpha\geq\beta$. Then there is a path of length $\beta$ edge-disjoint from $P$ attached to a vertex in $e_1\setminus\{x\}$
\end{lemma}

\begin{proof}
	Observe that by the length of $P$, $\alpha+\beta+1\leq m-3$. Assume to the contrary that we do not have a path of length $\beta$ leaving $e_1\setminus\{x\}$, edge-disjoint from $P$. Let $e_2$ be the first edge along $P$ after $e_1$. By the maximality of $P$ and our assumption, the longest path leaving a vertex in $e_1\cup e_2\setminus\{x\}$ is of length $\beta-1$. Thus by Lemma~\ref{twoedge}, applied to $e_1$ and $e_2$, $\alpha+\beta-1\geq m-4$, contradicting $\alpha+\beta+1\leq m-3$.	
\end{proof}

\subsection{Counting Lemmas}\label{section counting lemmas}

Now that we have established the minimal amount of branching in a Berge-$P_m$-saturated linear tree, we can begin counting the number of edges in a saturated tree $H$. Given an edge $e$ that is not in the center of $H$, these counting lemmas will give a lower bound on the number of edges in the component of $H-e$ that does not contain the center of $H$. As with the branching lemmas we need to separate into cases based on if $H$ contains a Berge-$P_{m-1}$ or not.

\subsubsection{Counting with a Berge-$\mathbf{P_{m-1}}$}

As with the branching lemmas, when $H$ contains Berge-$P_{m-1}$, we must consider different cases for $k=3$, $k=4$ and $k\geq 6$.

\begin{lemma}\label{k=3 P_{m-1} counting lemma} Let $H$ be a $3$-uniform Berge-$P_m$-saturated linear tree on $n\geq 4$ vertices. Let $e$ be an edge of $H$ and $P$ be a Berge-$P_{m-1}$ through $e$. 
Let $\alpha$ be the number of edges preceding $e$ in $P$ and $\beta$ the number of edges after $e$ attached at $v\in e$ with $\alpha\geq \beta$.
 Let $P'$ be the path on $\beta$ edges starting at $v$ contained in $P$. Let $X$ be the set of all vertices $x$ such that the path between $x$ and $v$ uses at least one edge of $P'$. Then 
		\[
		|E(X\cup \{v\})|\geq
		\begin{cases} 
		2^{\ell+1}-2&\text{ if }\beta=2\ell\\
		2^{\ell+1}+2^{\ell}-2&\text{ if }\beta=2\ell+1.
		\end{cases}
		\]
\end{lemma}

\begin{proof}
	We will proceed by strong induction on $\beta$. For $0 \leq \beta\leq 2$ the result holds simply counting the edges on $P'$. Now let $\beta\geq 3$ and assume that the result holds for all paths of length less than $\beta$.
	
	Let $e_1$ and $e_2$ be the first two edges of the path $P'$ starting at $v\in  e$, in that order. Then there is a Berge-$P_{m-1}$ containing $e_1$ and $e_2$ with $\alpha+1$ edges preceding $e_1$ and $\beta-2$ edges after $e_2$ with $\alpha+1>\beta-2$. Then by Lemma~\ref{k=3 P_{m-1} branching lemma} there is a path of length at least $\beta-2$ leaving $e_1$ or $e_2$, edge disjoint from the Berge-$P_{m-1}$. Then we have two paths of length $\beta-2$ leaving $e_1\cup e_2$, and $e_1$ and $e_2$ all in $X$. By induction, this gives us that if $\beta=2\ell$, then
	\[
	|X|\geq 2(2^\ell-2)+2=2^{\ell+1}-2,
	\]
	and if $\beta=2\ell+1$,
	\[
	|X|\geq 2(2^{\ell}+2^{\ell-1}-2)+2=2^{\ell+1}+2^{\ell}-2.
	\]
\end{proof}

\begin{lemma}\label{k=4 P_{m-1} counting lemma}
Let $H$ be a $4$-uniform Berge-$P_m$-saturated linear tree on $n\geq 5$ vertices. Let $e$ be an edge of $H$ and let $P$ be a Berge-$P_{m-1}$ that uses $e$. Assume there are $\alpha$ edges preceding $e$ in $P$ and $\beta$ edges following $e$ attached at $v\in e$ with $\alpha\geq \beta$. Let $P'$ be the path on $\beta$ edges starting at $v$ contained in $P$. Let $X$ be the set of all vertices $x$ such that the path between $x$ and $v$ uses at least one edge of $P'$. Then 
	\[
	|E(X\cup \{v\})|\geq
	\begin{cases} 
	2^{\ell+1}+2^\ell-3&\text{ if }\beta=3\ell\\
	2^{\ell+2}-3&\text{ if }\beta=3\ell+1\\
	2^{\ell+2}+2^\ell-3&\text{ if }\beta=3\ell+2.
	\end{cases}
	\]
\end{lemma}

\begin{proof}
	We will proceed by strong induction based on the value of $\beta$ modulo 3. For $0 \leq \beta\leq 3$ the result holds simply counting the edges on $P'$. Let $\beta \geq 4$ and assume the result holds for all paths of shorter length.
	
	Let $e_1$, $e_2$ and $e_3$ be the first three edges of $P'$ in that order. Then there is a Berge-$P_{m-1}$ containing $e_1$, $e_2$ and $e_3$ with $a+1$ edges preceding $e_1$ and $\beta-3$ edges after $e_3$ with $\alpha+1>\beta-3$. By Corollary~\ref{k=4 P_{m-1} branching corollary} there is a path of length at least $\beta-3$ leaving a vertex in $e_1\cup e_2\cup e_3$, edge disjoint from the Berge-$P_{m-1}$. Then $X$ contains two paths of length $\beta-3$ leaving $e_1\cup e_2\cup e_3$, and also the edges $e_1$, $e_2$ and $e_3$. By induction, this gives us that if $\beta=3\ell$, then
	\[
	|X|\geq 2(2^\ell+2^{\ell-1}-3)+3=2^{\ell+1}+2^\ell-3.
	\]
	If $\beta=3\ell+1$, then
	\[
	|X|\geq 2(2^{\ell+1}-3)+3=2^{\ell+2}-3.
	\]
	If $\beta=3\ell+2$, then
	\[
	|X|\geq 2(2^{\ell+1}+2^{\ell-1}-3)+3=2^{\ell+2}+2^\ell-3.
	\]
	
\end{proof}

\begin{lemma}\label{k>= 6 P_{m-1} counting lemma}
	Let $k\geq 6$. Let $H$ be a $k$-uniform Berge-$P_m$-saturated linear tree on $n\geq k+1$ vertices. Let $e$ be an edge and $P$ a Berge-$P_{m-1}$ in $H$ that uses $e$. Assume there are $\alpha$ edges preceding $e$ in $P$ and $\beta$ edges following $e$ attached at $v\in e$ with $\alpha\geq\beta$. Let $P'$ be the path on $\beta$ edges starting at $v$ contained in $P$. Let $X$ be the set of all vertices $x$ such that the path between $x$ and $v$ uses at least one edge of $P'$. Then 
	\[
	|E(X\cup \{v\})|\geq
	\begin{cases} 
	2^{\ell+1}-2&\text{ if }\beta=2\ell\\
	2^{\ell+1}+2^{\ell-1}-2&\beta=2\ell+1.
	\end{cases}
	\]
\end{lemma}

\begin{proof}
	For $1 \leq \beta\leq 3$ the result holds simply counting the edges on $P'$. Let $\beta=4$. We want to show that we have at least $6$ edges. Let $e_1$, $e_2$, $e_3$ and $e_4$ be the edges in $P'$ and assume there is at most one more edge, $e_5$, incident with these edges. Since $\beta=4$, $e_5$ is not incident to $e_4$. If $e_5$ is incident to $w$ in $e_1$ or $e_3$, or if there is no fifth edge $e_5$, then adding the edge $e'$ which consists of $k$ vertices of degree $1$ in $(e_1\cup e_2)\setminus \{w\}$ does not create a Berge-$P_m$ in $H+e'$. If $e_5$ is incident to a degree $1$ vertex, $w$, in $e_2$ then adding the edge $e'=e_5\setminus \{w\}\cup \{u\}$ where $u$ is another degree $1$ vertex in $e_2$ does not create a Berge-$P_m$ in $H+e'$ contradicting the assumption that $H$ is Berge-$P_m$-saturated. Thus we have at least $6$ edges, so the result holds for $\beta=4$.
	
	Let us proceed by strong induction on $\beta$ based on parity. Let $\beta\geq 5$ and assume the result holds for paths of shorter length. 
	
	Let $e_1$, $e_2$ and $e_3$ be the first three edges in $P'$. There is a path of length $\beta-3$ leaving $e_3$ as shown in Figure~\ref{fig:k>=6Counting}. Then by Lemma~\ref{k>= 6 P_{m-1} branching lemma} there is either a path of length $\beta-3$ leaving $e_2$ or there are at least $k-2\geq 4$ paths of length $\beta-3$ leaving $(e_1\cup e_3)\setminus (e_2\cup \{v\})$. If there are $k-2$ paths of length $\beta-3$ leaving $(e_1\cup e_3)\setminus (e_2\cup \{v\})$, then by our inductive hypothesis we have
	
	\[
	|E(X\cup \{v\})|\geq
	\begin{cases} 
	(k-2)(2^{\ell-1}+2^{\ell-3}-2)+3& \text{ if }\beta=2\ell \\
	(k-2)(2^\ell-2)+3& \text{ if }\beta=2\ell+1. 
	\end{cases}
	\]
	Note that since $k\geq 6$ and $\ell\geq 2$, we have
	\[
	|E(X\cup \{v\})|\geq
	\begin{cases} 
	2^{\ell+1}+2^{\ell-1}-5& \text{ if }\beta=2\ell \\
	2^{\ell+2}-5& \text{ if }\beta=2\ell+1. 
	\end{cases}
	\]
	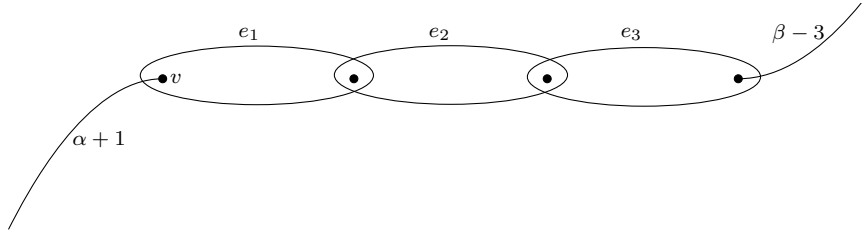
\begin{figure}  [h]
		\begin{center}  
			\begin{tikzpicture}[line cap=round,line join=round,>=triangle 45,x=1.0cm,y=1.0cm]
			\clip(-2,-2) rectangle (10,1);
			\draw [rotate around={0.:(1.3,0.045)}] (1.30,0.045) ellipse (1.55cm and 0.39cm); 
			\draw [rotate around={0.:(3.88,0)}] (3.88,0.05) ellipse (1.55cm and 0.39cm); 
			\draw [rotate around={0.:(6.446,0.025)}] (6.446,0.025) ellipse (1.55cm and 0.39cm); 
			\begin{scriptsize}
			\draw[color=black] (1.2,.6) node {$e_1$};
			\draw[color=black] (3.73,.6) node {$e_2$};
			\draw[color=black] (6.28,.6) node {$e_3$};
			\draw [fill=black] (2.59,0) circle (1.5pt); 
			\draw [fill=black] (5.16,0.0) circle (1.5pt); 
			\draw [fill=black] (0.048,0.0) circle (1.5pt); 
			\draw[color=black] (0.226,0.0) node {$v$};
			\draw [fill=black] (7.7,0.0) circle (1.5pt); 
			
			\draw (0.048,0) parabola (-2,-2);\draw[color=black] (-.8,-.8) node {$\alpha+1$}; 
			\draw (7.7,0)parabola  (10,2);\draw[color=black] (8.5,.6) node {$\beta-3$};
			\end{scriptsize}
			\end{tikzpicture}
		\end{center}  
		\caption{\label{fig:k>=6Counting}The edges $e_1$, $e_2$ and $e_3$ as described in Lemma~\ref{k>= 6 P_{m-1} counting lemma}.} 
	\end{figure}

If we do not have these $k-2$ paths, then there is a path of length $\beta-3$ leaving $e_2$. If this path cannot be extended to a path of length $\beta-2$, then the edges in it are in no Berge-$P_{m-1}$ so by Lemma~\ref{no P_{m-1} counting lemma} this path contributes at least $2^{\beta-3}-1$ edges. In addition to this, we have the path of length $\beta-3$ leaving $e_3$. Then by our inductive hypothesis, we have
\[
|E(X\cup \{v\})|\geq
\begin{cases}
2^{2\ell-3}-1+2^{\ell-1}+2^{\ell-3}-2+3=2^{2\ell-3}+2^{\ell-1}+2^{\ell-3}& \text{ if }\beta=2\ell\\
2^{2\ell-2}-1+2^\ell-2+3=2^{2\ell-2}+2^\ell&\text{ if }\beta=2\ell+1. 
\end{cases}
\]

If instead the path can be extended to a path of length $\beta-2$, then this path and the path of length $\beta-3$ give us two paths of length $\beta-2$ leaving $e_2$, which give us
\[
|E(X\cup \{v\})|\geq
\begin{cases}
2(2^\ell-2)+2=2^{\ell+1}-2&\text{ if }\beta=2\ell\\ 
2(2^\ell+2^{\ell-2}-2)+2=2^{\ell+1}+2^{\ell-1}-2&\text{ if }\beta=2\ell+1.
\end{cases}
\]
It is easy to check that this final case is minimal for all $\beta\geq 5$, giving us the result.
\end{proof}

\subsubsection{Counting with no Berge-$\mathbf{P_{m-1}}$}

Since there is only one branching lemma for all uniformities when $H$ does not contain a Berge-$P_{m-1}$, there is also only one counting lemma for this case.

\begin{lemma}\label{no P_{m-1} counting lemma}
	Let $k\geq 3$. Let $H$ be a $k$-uniform Berge-$P_m$-saturated linear tree on $n\geq k+1$ vertices. Let $e$ be an edge and $P$ be a Berge-$P_{m-2}$ that is the longest path using $e$. Assume there are $\alpha$ edges preceding $e$ in $P$ and $\beta$ edges following $e$ attached at $v\in e$ with $\alpha\geq \beta$. Let $P'$ be the path on $\beta$ edges starting at $v$ contained in $P$ and let $e_1$ be the first edge of $P'$(so $v\in e_1$.). Let $X$ be the set of all vertices $x$ such that the path between $x$ and $v$ uses at least one edge of $P'$. If $P'$ is the longest path that does not use edge $e$ but contains $e_1$, then the number of edges in $X\cup\{v\}$ is at least $2^\beta-1$.
\end{lemma}

\begin{proof}
	Let us proceed via induction. For $\beta=1$ the result is trivial. Let $\beta\geq 1$ and assume the result is true for paths of length $\beta-1$. Then, by Lemma~\ref{no P_{m-1} branching structure}, we have a path edge-disjoint from $P'$ attached to a vertex in $e_1$ of length $\beta-1$. By our inductive hypothesis this path contributes at least $2^{\beta-1}-1$ edges. Add this to the $2^{\beta-1}-1$ edges given by the path $P'-e_1$ of length $\beta-1$ and the one edge $e_1$, we get that the number of edges is at least $2(2^{\beta-1}-1)+1=2^\beta-1$.
\end{proof}

\subsection{Central Structure and Final Count}\label{section central structure}

Using the branching and counting lemmas of the previous subsections, we now determine the central structure of a Berge-$P_{m-1}$-saturated linear tree with a minimum number of edges. We can then use this central structure and the counting lemmas to get a lower bound on the size of these trees. In order to rule out the case, we first consider when $H$ has no Berge-$P_{m-1}$, and show that this is not optimal. We then give minimal edge counts for general Berge-$P_m$-saturated trees.

\subsubsection{Final count with no Berge-$\mathbf{P_{m-1}}$}

The following lemma gives a lower bound for a tree $H$ with no Berge-$P_{m-1}$. We will see that this case is far from optimal.

\begin{lemma}\label{no P_{m-1} final count lemma}
	Let $m\geq 10$ and $k\geq 3$. Let $H$ be a Berge-$P_m$ saturated linear tree on at least $k+1$ vertices and let $H$ contain no Berge-$P_{m-1}$. Then 
	\[
	|E(H)|\geq
	\begin{cases} 
	3(2^{\ell-2})-2&\text{ if }m=2\ell\\
	2^\ell-2&\text{ if }m=2\ell+1.
	\end{cases}
	\]
\end{lemma}

\begin{proof}
	By Observation~\ref{Pm-2} there exists some longest path $P$ that is a Berge-$P_{m-2}$ in $H$.
	
	First consider the case where $m=2\ell$ is even. Then the longest path is of odd length $m-3=2\ell-3$. Let $e$ be the central edge of this path.  By Lemma~\ref{no P_{m-1} branching structure} with $\alpha=\beta=\ell-2$, we have three paths of length $\ell-2$ leaving $e$. By Lemma~\ref{no P_{m-1} counting lemma} the existence of each path guarantees at least $2^{\ell-2}-1$ edges, so adding these to the one edge $e$, we have 
	\[
	|E(H)|\geq 3(2^{\ell-2}-1)+1=3(2^{\ell-2})-2.
	\]

	Now consider the case when $m=2\ell+1$ is odd. Then the longest path is of even length $m-3=2\ell-2$. Let $e_1$ and $e_2$ be the two middle edges of $P$. These each have a path of length $\ell-2$ leaving them. By Lemma~\ref{no P_{m-1} branching structure}, there are either two paths of length at least $\ell-2$  edge-disjoint from $P$ attached to $e_1$ and $e_2$, call this situation (i) (see Figure~\ref{noP_m-1Thm2}), or one path of length at least $\ell-2$ attached to $e_1\cap e_2$, call this situation (ii) (see Figure~\ref{noP_m-1Thm3}).
	
		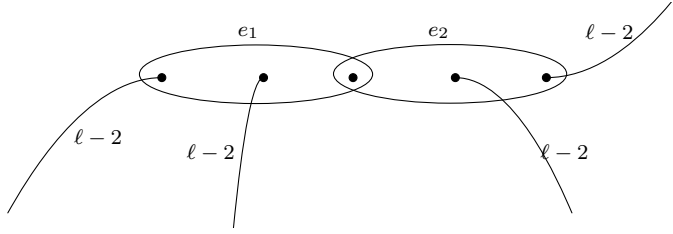
\begin{figure}[h]  
		\begin{center}  
			\begin{tikzpicture}[line cap=round,line join=round,>=triangle 45,x=1.0cm,y=1.0cm]
			\clip(-2,-2) rectangle (8,1);
			\draw [rotate around={0.:(1.3,0.045)}] (1.30,0.045) ellipse (1.55cm and 0.39cm); 
			\draw [rotate around={0.:(3.88,0)}] (3.88,0.05) ellipse (1.55cm and 0.39cm); 
			
			\begin{scriptsize}
			\draw[color=black] (1.2,.6) node {$e_1$};
			\draw[color=black] (3.73,.6) node {$e_2$};
			\draw [fill=black] (0.048,0.0) circle (1.5pt); 
			\draw [fill=black] (2.59,0) circle (1.5pt); 
			\draw [fill=black] (5.16,0.0) circle (1.5pt); 
			\draw [fill=black] (1.4,0.0) circle (1.5pt); 
			\draw [fill=black] (3.95,0.0) circle (1.5pt); 
			
			\draw (0.048,0) parabola (-2,-1.8);\draw[color=black] (-.8,-.8) node {$\ell-2$}; 
			\draw (5.16,0)parabola  (7.5,2);\draw[color=black] (6,.6) node {$\ell-2$};
			\draw (1.4,0) parabola (1,-2);\draw[color=black] (.7,-1) node {$\ell-2$}; 
			\draw (3.95,0)parabola  (5.5,-1.8);\draw[color=black] (5.4,-1) node {$\ell-2$};
			\end{scriptsize}
			\end{tikzpicture}
		\end{center}  
		\caption{Situation (i) from Lemma~\ref{no P_{m-1} final count lemma} when $m=2\ell+1.$} \label{noP_m-1Thm2}
	\end{figure}  
	
	If we are in situation (i), Lemma~\ref{no P_{m-1} counting lemma} guarantees that each of these four paths contribute at least $2^{\ell-2}-1$ edges. These paths along with the two edges $e_1$ and $e_2$ gives us that
	\[
	|E(H)|\geq 4(2^{\ell-2}-1)+2=2^\ell-2.
	\] 

	\begin{figure} [h]
	\begin{center}  
		\begin{tikzpicture}[line cap=round,line join=round,>=triangle 45,x=1.0cm,y=1.0cm]
		\clip(-2,-2) rectangle (8,1);
		\draw [rotate around={0.:(1.3,0.045)}] (1.30,0.045) ellipse (1.55cm and 0.39cm); 
		\draw [rotate around={0.:(3.88,0)}] (3.88,0.05) ellipse (1.55cm and 0.39cm); 
		
		\begin{scriptsize}
		\draw[color=black] (1.2,.6) node {$e_1$};
		\draw[color=black] (3.73,.6) node {$e_2$};
		\draw [fill=black] (0.048,0.0) circle (1.5pt); 
		\draw [fill=black] (2.59,0) circle (1.5pt); 
		\draw [fill=black] (5.16,0.0) circle (1.5pt); 
		
		\draw (0.048,0) parabola (-2,-1.8);\draw[color=black] (-.8,-.8) node {$\ell-2$}; 
		\draw (5.16,0)parabola  (7.5,2);\draw[color=black] (6,.6) node {$\ell-2$};
		\draw (2.59,0) parabola (3,-2);\draw[color=black] (3.3,-1) node {$\geq\ell-2$}; 
		\end{scriptsize}
		\end{tikzpicture}
	\end{center}  
	\caption{Situation (ii) from Lemma~\ref{no P_{m-1} final count lemma} when $m=2\ell+1$.} \label{noP_m-1Thm3}
\end{figure}
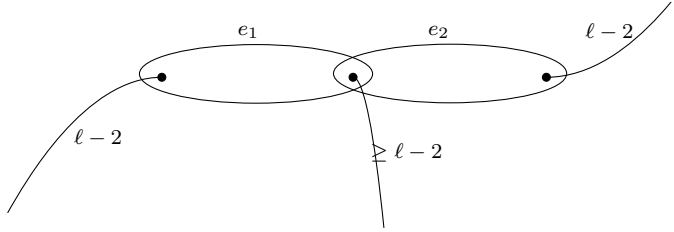   
	
	If we are in situation (ii), Lemma~\ref{edge in m-3} gives us that the edges in the path, $P'$ attached to $e_1\cap e_2$ must be contained in a Berge-$P_{m-2}$, so $P'$ must have length at least $\ell-1$. Now by Lemma~\ref{no P_{m-1} counting lemma}, the two paths of length $\ell-2$ contribute $2^{\ell-2}-1$ edges while the one path of length $\ell-1$ contributes $2^{\ell-1}-1$ edges. Thus we have that
	\[
	|E(H)|\geq 2(2^{\ell-2}-1)+2^{\ell-1}-1+2=2^\ell-1>2^\ell-2.
	\]
\end{proof}

\subsubsection{Final Count with a Berge-$\mathbf{P_{m-1}}$}

This section provides lower bounds of edge counts for general Berge-$P_m$-saturated trees for uniformities $k=3$, $k=4$, and $k\geq 6$.

\begin{theorem}\label{k=3 P_{m-1} final count theorem}
		Let $H$ be a $3$-uniform Berge-$P_m$-saturated linear tree with at least $4$ vertices. For $m\geq 10$, let $m=4s+r$ with $1\leq r\leq 4$. Then
		\[
		|E(H)|\geq (3+r)2^s-5.
		\]
\end{theorem}

\begin{proof}
	First, observe that if $H$ does not contain a Berge-$P_{m-1}$, then Lemma~\ref{no P_{m-1} final count lemma} implies our result. Therefore, we will assume that $H$ does contain a Berge-$P_{m-1}$.
	
	Let $m=2\ell$. Let $P$ be a Berge-$P_{m-1}$ in $H$. Then $P$ is of length $2\ell-2$. Let $e_1$ and $e_2$ be the two central edges of $P$. Then there are $\ell-2$ edges preceding $e_1$ and following $e_2$, so by Lemma~\ref{k=3 P_{m-1} branching lemma}, we have that there must be a third path of length at least $\ell-2$ attached to a vertex in $e_1$ or $e_2$, say $e_1$. Then just considering paths away from $e_1$, we have a path of length $\ell-1$ and two paths of length $\ell-2$, all edge disjoint and not using the edge $e_1$. 
	
	If $m=4s+4$, then $\ell=2s+2$, so $\ell-1=2s+1$ is odd and $\ell-2=2s$ is even. Then by Lemma~\ref{k=3 P_{m-1} counting lemma}, from the path of length $\ell-1$ we have at least $2^{s+1}+2^{s}-2$ edges and from the two paths of length $\ell-2$, we have at least $2(2^{s+1}-2)$ edges. In addition to these, we have $e_1$. Thus we get that
	\[
	|E(H)|\geq 1+2^{s+1}+2^{s}-2+2(2^{s+1}-2)=7(2^s)-5
	\]
	
	If $m=4s+2$, then $\ell=2s+1$, so $\ell-1=2s$ is even and $\ell-2=2(s-1)+1$ is odd. Then by Lemma~\ref{k=3 P_{m-1} counting lemma}, from the path of length $\ell-1$ we have at least $2^{s+1}-2$ edges and from the two paths of length $\ell-2$, we have at least $2^{s}+2^{s-1}-2$ edges. In addition to these, we have $e_1$. Thus we get that
	\[
	|E(H)|\geq 1+2^{s+1}-2+2(2^{s}+2^{s-1}-2)=5(2^s)-5.
	\]
	
	Now let $m=2\ell+1$. Let $P$ be a Berge-$P_{m-1}$ in $H$. Then $P$ is of length $2\ell-1$. Let $e_1$ be the central edge in $P$, and let $e_2$ be the edge immediately after $e_1$ in $P$, Then there are $\ell-1$ edges preceding $e_1$ and $\ell-2$ edges after $e_2$, so by Lemma~\ref{k=3 P_{m-1} branching lemma} there is another path of length at least $\ell-2$ coming from a vertex in either $e_1$ or $e_2$. 
	
	If this path comes from the edge $e_2$, then $e_2$ has a path of length $\ell$ and two paths of length $\ell-2$ coming from it, all edge disjoint and not using the edge $e_2$. Call this situation (i).
	
	If instead this path is coming form $e_1$, then $e_1$ has a path of length at least $\ell-2$ and two paths of length $\ell-1$ coming from it, all edge disjoint and not using the edge $e_1$. It may end up that there are actually three paths of length $\ell-1$ leaving $e_1$. If this is the case, call it situation (ii).
	
	If we are not in situation (i) or (ii), then there are two paths of length exactly $\ell-1$ leaving $e_1$ and the third longest path is of length exactly $\ell-2$. Observe that the edges of this last path are not in any Berge-$P_{m-1}$ in $H$. Call this situation (iii).
	
	If $m=4s+1$ and we have situation (i), then $\ell=2s$ and $\ell-2=2(s-1)$ are both even. Then by Lemma~\ref{k=3 P_{m-1} counting lemma}, from the path of length $\ell$ we have at least $2^{s+1}-2$ edges and from the two paths of length $\ell-2$, we have at least $2^{s}-2$ edges. In addition to these, we have $e_2$. Thus we get that
	\[
	|E(H)|\geq 1+2^{s+1}-2+2(2^{s}-2)=4(2^{s})-5
	\]

	If $m=4s+1$ and we have situation (ii), then
	 $\ell-1=2(s-1)+1$ is odd. Then by Lemma~\ref{k=3 P_{m-1} counting lemma} from each path of length $\ell-1$, we have at least $2^s+2^{s-1}-2$ edges, and we also have the edge $e_1$. Thus
	 \[
	 |E(H)|\geq 1+3(2^s+2^{s-1}-2)=2^{s+1}=2^{s+2}+2^{s-1}-5
	 \]
	 
	 If $m=4s+1$ and we have situation (iii), then $\ell-1=2(s-1)+1$ is odd. Then by Lemma~\ref{k=3 P_{m-1} counting lemma} from each path of length $\ell-1$, we have at least $2^s+2^{s-1}-2$ edges. Furthermore, since the path of length $\ell-2=2s-2$ is not in any Berge-$P_{m-1}$, by Lemma~\ref{no P_{m-1} counting lemma} we have at least $2^{2s-2}-1$ edges. We also have the edge $e_1$, giving us
	 \[
	 |E(H)|\geq 1+2(2^s+2^{s-1}-2)+2^{2s-2}-1=2^{2s-2}+2^{s+1}+2^s-4.
	 \]
	 
	 Thus for $m=4s+1$, since $m\geq 10$, implying $s\geq 2$, we get the lowest bound from situation (i), giving our result.

	Now let us consider $m=4s+3$. If we have situation (i), then $\ell=2s+1$ and $\ell-2=2(s-1)+1$ are odd. Then by Lemma~\ref{k=3 P_{m-1} counting lemma}, from the path of length $\ell$ we have at least $2^{s+1}+2^{s}-2$ edges and from each path of length $\ell-2$, we have at least $2^{s}+2^{s-1}-2$ edges. In addition to these, we have $e_2$. Thus we get that
	\[
	|E(H)|\geq 1+2^{s+1}+2^{s}-2+2(2^s+2^{s-1}-2)=6(2^{s})-5
	\]
	
	If $m=4s+3$ and we have situation (ii), $\ell-1=2s$ is even.  Then by Lemma~\ref{k=3 P_{m-1} counting lemma}, we have at least $2^{s+1}-2$ edges from each of the three paths of length $\ell-1$. In addition to these, we have $e_1$. Thus we get that
	\[
	|E(H)|\geq 1+3(2^{s+1}-2)=6(2^{s})-5
	\]
	
	 If $m=4s+3$ and we have situation (iii), then $\ell-1=2s$ is even. Then by Lemma~\ref{k=3 P_{m-1} counting lemma} from each path of length $\ell-1$, we have at least $2^{s+1}-2$ edges. Furthermore, since the path of length $\ell-2=2s-1$ is not in any Berge-$P_{m-1}$, by Lemma~\ref{k=3 P_{m-1} counting lemma} we have at least $2^{2s-1}-1$ edges. We also have the edge $e_1$, giving us
	\[
	|E(H)|\geq 1+2(2^{s+1}-2)+2^{2s-1}-1=2^{2s-1}+2^{s+2}-4.
	\]

	Thus in all cases we have that for $m=4s+3$, since $m\geq 10$, and consequently $s\geq 2$, $|E(H)|\geq 6(2^{s})-5$.
	
\end{proof}

\begin{theorem}\label{k=4 P_{m-1} final count theorem}
	Let $H$ be a $4$-uniform Berge-$P_m$ saturated linear tree with at least $5$ vertices. For $m\geq 10$, let $0\leq r\leq 5$ such that $m=6s+r$. Then
	\[
	|E(H)|\geq (6+r)2^s-8.
	\]
\end{theorem}

\begin{proof}
	If $H$ does not contain a Berge-$P_{m-1}$, then Lemma~\ref{no P_{m-1} final count lemma} implies our result, so we will assume that $H$ does have a Berge-$P_{m-1}$.
	
	Now, whenever counting edges, if we have a path away from the center of length exactly $\beta$ whose edges are not in any Berge-$P_{m-1}$, then by Lemma~\ref{no P_{m-1} counting lemma} we get that this path contributes at least  $2^{\beta}-1$. If instead this path was length $\beta+1$ with edges in some Berge-$P_{m-1}$, then by Lemma~\ref{k=4 P_{m-1} counting lemma}, the path contributes $2^{\ell+2}-3$ edges if $\beta=3\ell$, $2^{\ell+2}+2^\ell-3$ edges if $\beta=3\ell+1$ and $2^{\ell+2}+2^{\ell+1}-3$ edges if $\beta=3\ell+2$. Observe that if $\beta\geq 2$, we get a lower or equal count in all cases with the path of length $\beta+1$. Thus any time we have a long path in no Berge-$P_{m-1}$, we are justified in counting its contribution as if it were longer and in a Berge-$P_{m-1}$. This fact is crucial to the argument and will be used many times, so we will refer to the argument in this paragraph as ($\ast$).
	
	Now, let us assume $m=2\ell$ is even. Let $P$ be a Berge-$P_{m-1}$ in $H$. Note that $P$ has even length $2\ell-2$. Let $e_1, e_2, e_3$ and $e_4$ be the four central edges of $P$, appearing in that order. Then there are paths of length $\ell-3$ leaving some vertex in $e_1$ and some vertex in $e_4$. Applying Corollary~\ref{k=4 P_{m-1} branching corollary} on edges $e_2, e_3$ and $e_4$, we get that there is a path of length at least $\ell-3$ leaving a vertex in $e_2\cup e_3\cup e_4$.
	
	If this path leaves $e_4\setminus e_3$, then applying Corollary~\ref{k=4 P_{m-1} branching corollary} to $e_1$, $e_2$ and $e_3$, we get another path of length $\ell-3$ away from center. Then we have four paths of length $\ell-3$ and four central edges. Call this situation (i).
	
	If instead this path leaves a vertex in the symmetric difference $e_2\triangle e_3$, then we may assume it is of length $\ell-2$ since $\ell\geq 5$. Assume without loss of generality this path leaves a vertex in $e_3\setminus e_2$. Let $e_0$ be the edge preceding $e_1$ in $P$. Then $e_0$ has a path of length $\ell-4$ leaving it. Applying Corollary~\ref{k=4 P_{m-1} branching corollary} to edges $e_0$, $e_1$ and $e_2$ gives us a second path of length at least $\ell-4$. Instead of considering $e_4$ with a path of length $\ell-3$ away from $e_4$, we will instead consider this as a path of length $\ell-2$ away form $e_3$. Then we have two paths of length $\ell-2$, two paths of length $\ell-4$ and four central edges. Call this situation (ii).
	
	Finally if the path leaves the vertex in $e_2\cap e_3$, the path must be of length at least $\ell-2$ for otherwise the edges would not be in a Berge-$P_{m-2}$, contradicting Lemma~\ref{edge in m-3}. Now, let $e_0$ be the edge preceding $e_1$ in $P$ and $e_5$ be the edge after $e_4$ in $P$. Then each of these edges has paths of length $\ell-4$ leaving them. Further, let $e_1^*$, $e_2^*$ and $e_3^*$ be the first three edges of the path leaving the vertex in $e_2\cap e_3$. Then there is a path of length at least $\ell-5$ leaving $e_3^*$. Applying Corollary~\ref{k=4 P_{m-1} branching corollary} to the three triplets $e_0, e_1, e_2$; $e_3,e_4,e_5$ and $e_1^*,e_2^*,e_3^*$ gives us two new paths of length $\ell-4$ and a path of length $\ell-5$. Observe that if the paths of length $\ell-5$ are exactly length $\ell-5$, then their edges are in no Berge-$P_{m-1}$ since the path is at distance at most $3$ from the center, so the path, three edges, and a path of length $\ell-1$ gives us a length of at most $\ell-5+3+\ell-1=2\ell-3=m-3$. Thus if $m\geq 14$, $\ell-5\geq 2$, and by ($\ast$) we may count these paths as if they are length $\ell-4$. In this case, we have nine edges in the center with six paths of length at least $\ell-4$. Call this situation (iii). 
	
	If in this same setup though, $m=10$ or $m=12$, we cannot assume this path is of length $\ell-4$. In this case though, we still have the nine central edges and four paths of length $\ell-4\geq 1$. This gives us $13$ edges, which for both $m=10$ and $m=12$ exceeds our result, so we are safe in assuming that in situation (iii), $m\geq 14$.
	
	These three situations exhaust all possibilities for paths given by the first use of Corollary~\ref{k=4 P_{m-1} branching corollary}.
	
	In each of these situations, we can use Lemma~\ref{k=4 P_{m-1} counting lemma} to count how many edges are guaranteed by each path away from the center. To do so, we need to consider each situation in terms of the residue of $m$ modulo $6$. Let $m=6s+r$. Then for even $r$, the edge counts for each situation is summarized in the table below:
	
	\begin{center}
	{\small\setlength{\extrarowheight}{2pt}\begin{tabular}{|c|c|c|c|}
		\hline
		r & Situation & Calculation &Edge Count\\
		\hline
		0 & i & $4(2^{s}+2^{s-1}-3)+4$ & $6(2^s)-8$\\
		0 & ii & $2(2^{s+1}-3)+2(2^s+2^{s-2}-3)+4$ &$6(2^s)+2^{s-1}-8$\\
		0 & iii & $6(2^s+2^{s-2}-3)+9$ &$7(2^s)+2^{s-1}-9$\\
		\hline
		2 & i & $4(2^{s+1}-3)+4$ & $8(2^{s})-8$\\
		2 & ii & $2(2^{s+1}+2^{s-1}-3)+2(2^s+2^{s-1}-3)+4$ &$8(2^{s})-8$\\
		2 & iii & $6(2^s+2^{s-1}-3)+9$ &$9(2^{s})-9$\\
		\hline
		4 & i & $4(2^{s+1}+2^{s-1}-3)+4$ &$10(2^{s})-8$\\
		4 & ii & $2(2^{s+1}+2^s-3)+2(2^{s+1}-3)+4$ &$10(2^{s})-8$\\
		4 & iii & $6(2^{s+1}-3)+9$ &$12(2^{s})-9$\\
		\hline
	\end{tabular}}
	\end{center}

	Taking the minimal value for each even $r$ gives our result for even $m$.
	
	Now let us consider odd $m=2\ell+1$. Let $P$ be a Berge-$P_{m-1}$ in $H$. Note that $P$ has odd length $2\ell-1$. Let $e_1, e_2, e_3$, $e_4$ and $e_5$ be the five central edges of $P$, appearing in that order. Then there are paths of length $\ell-3$ leaving some vertex in $e_1$ and some vertex in $e_5$. Applying Corollary~\ref{k=4 P_{m-1} branching corollary} to edges $e_3, e_4$ and $e_5$, we get that there is a path of length at least $\ell-3$ leaving a vertex in $e_3\cup e_4\cup e_5$. 
	
	If this path leaves a vertex in $e_5$, then we can apply Corollary~\ref{k=4 P_{m-1} branching corollary} on edges $e_2$, $e_3$ and $e_4$ to get a path of length at least $\ell-2$ leaving a vertex in $e_2\cup e_3\cup e_4$. Observe that this path is not the same path found on the previous use of the corollary since the corollary guarantees that the path attaches at a vertex different from the two end vertices. If we then consider the edge $e_1$ and the path away from it of length $\ell-3$ as a path of length $\ell-2$ away from $e_2$, we have two paths of length $\ell-2$, two paths of length $\ell-2$ and four central edges. Call this situation (iv).
	
	If the path instead leaves a vertex in $e_4\setminus e_3$, then since $\ell-3>2$, we can assume this path is of length $\ell-2$. Applying Corollary~\ref{k=4 P_{m-1} branching corollary} to edges $e_1$, $e_2$ and $e_3$ gives us another path of length $\ell-3$. If we then consider the edge $e_5$ and the path away from it of length $\ell-3$ as a path of length $\ell-2$ away from $e_4$, we get a situation identical to situation (iv).
	
	If the path leaves a vertex in $e_3$, then it is of length at least $\ell-2$ since otherwise the edges in it would be in no Berge-$P_{m-2}$, contradicting Lemma~\ref{edge in m-3}. We can further assume by ($\ast$), this path has length $\ell-1$ since otherwise it would be in no Berge-$P_{m-1}$ and $\ell-2>2$. If we then consider the edges $e_2$, $e_1$ and the path of length $\ell-3$ away from $e_1$ as a path of length $\ell-1$ away from $e_3$, and similarly consider the edges $e_4$, $e_5$ and the path of length $\ell-3$ away from $e_5$ as a path of length $\ell-1$ away from $e_3$, then have three paths of length $\ell-1$ and one central edge. Let this be situation (v).
	
	This exhausts all possibilities for the location of the path given by the first use of Corollary~\ref{k=4 P_{m-1} branching corollary}.
	
	Now, let $m=6s+r$. Similar to the even case, we can count the edges in each of these situations using Lemma~\ref{k=4 P_{m-1} counting lemma} to count the edges given by each path away from center. The counts are summarized in the table below for odd values of $r$:
	
	\begin{center}
	{\small\setlength{\extrarowheight}{2pt}\begin{tabular}{|c|c|c|c|}
		\hline
		r & Situation & Calculation &Edge Count\\
		\hline 
		1 & iv & $2(2^{s+1}-3)+2(2^s+2^{s-1}-3)+4$ &$7(2^{s})-8$\\
		1 & v & $3(2^{s+1}+2^{s-1}-3)+1$ &$7(2^{s})+2^{s-1}-8$\\
		\hline
		3 & iv & $2(2^{s+1}+2^{s-1}-3)+2(2^{s+1}-3)+4$ &$9(2^{s})-8$\\
		3 & v & $3(2^{s+1}+2^{s}-3)+1$ &$9(2^{s})-8$\\
		\hline
		5 & iv & $2(2^{s+1}+2^s-3)+2(2^{s+1}+2^{s-1}-3)+4$ &$11(2^{s})-8$\\
		5 & v & $3(2^{s+2}-3)+1$ &$12(2^{s})-8$\\
		\hline
	\end{tabular}}
	\end{center}
	
	Taking the minimal value for each $r$ gives us our result. This completes the proof.
\end{proof}

\begin{theorem}\label{k>=6 P_{m-1} final count theorem}
	Let $k\geq 6$ and $m\geq 10$. Let $H$ be a Berge-$P_m$ saturated linear tree on at least $k+1$ vertices. Then 
	\[
	|E(H)|\geq
	\begin{cases} 
	2^{s+1}+2^s+2^{s-1}+2^{s-2}-6&\text{ if }m=4s,\\
	2^{s+2}+2^{s-1}-6&\text{ if }m=4s+1,\\
	2^{s+2}+2^{s}-6&\text{ if }m=4s+2,\\
	2^{s+2}+2^{s+1}+2^{s-1}-6&\text{ if }m=4s+3.
	\end{cases}
	\]
\end{theorem}

\begin{proof}
	If $H$ contains no Berge-$P_{m-1}$ then we are done by Lemma~\ref{no P_{m-1} final count lemma}, so assume $H$ contains a Berge-$P_{m-1}$, say $P$.
	
	Whenever we count edges, if we have a path away from the center of length exactly $\beta$ whose edges are not in any Berge-$P_{m-1}$, then by Lemma~\ref{no P_{m-1} counting lemma} we get that this path contributes at least  $2^{\beta}-1$ edges. If instead this path was length $\beta+1$ with edges in some Berge-$P_{m-1}$, then by Lemma~\ref{k>= 6 P_{m-1} counting lemma}, the path contributes $2^{\ell+1}+2^{\ell-1}-2$ edges if $\beta=2\ell$ or $2^{\ell+2}-2$ edges if $\beta=2\ell+1$. Observe that if $\beta\geq 2$, we get a lower or equal count in all cases with the path of length $\beta+1$. Thus any time we have a long path in no Berge-$P_{m-1}$, we are justified in counting its contribution as if it were longer and in a Berge-$P_{m-1}$. This argument will be used many times, so we will refer to it as ($\star$).
	
	Let $m=4s+r$. Consider first when $m=2\ell$ is even, so $r=0$ or $r=2$. Then the longest path in $H$ is of even length $2\ell-2$. Let $e_1, e_2, e_3$ and $e_4$ be the four central edges of $P$, appearing in that order. Then there are paths of length $\ell-3$ leaving some vertex in $e_1$ and some vertex in $e_4$. Applying Lemma~\ref{k>= 6 P_{m-1} branching lemma} on edges $e_2, e_3$ and $e_4$, we get that there is either a path of length at least $\ell-3$ leaving $e_3$ or there are $k-1$ paths of length $\ell-3$ leaving $(e_2\cup e_4)\setminus e_3$. 
	
	If we have $k-1$ paths of length $\ell-3$ leaving $(e_2\cup e_4)\setminus e_3$, we have a set of three edges with $k-1$ paths of length $\ell-3$ leaving them. We will call this situation (i).
	
	If we are not in situation (i), then there must be a path of length at least $\ell-3$ leaving $e_3$. If this path attaches to $P$ via the vertex in $e_2\cap e_3$, then it is of length at least $\ell-2$ since otherwise the edges in this path would be in no Berge-$P_{m-2}$, which contradicts Lemma~\ref{edge in m-3}. Furthermore, by ($\star$), we can assume for counting purposes that this path is actually of length $\ell-1$. Notice that in this case there are three paths of length $\ell-1$ all leaving the vertex in $e_2\cap e_3$. Call this situation (ii) (see Figure~\ref{saturationNumberEveni}).
	
	\begin{figure}  [h]
		\begin{center}  
			\begin{tikzpicture}[line cap=round,line join=round,>=triangle 45,x=1.0cm,y=1.0cm]
			\clip(-2,-2) rectangle (13,1);
			\draw [rotate around={0.:(1.3,0.0)}] (1.30,0.0) ellipse (1.55cm and 0.39cm); 
			\draw [rotate around={0.:(3.88,0)}] (3.88,0.0) ellipse (1.55cm and 0.39cm); 
			\draw [rotate around={0.:(6.446,0.0)}] (6.446,0.0) ellipse (1.55cm and 0.39cm); 
			\draw [rotate around={0.:(8.9,0.0)}] (8.9,0.0) ellipse (1.55cm and 0.39cm); 
			\begin{scriptsize}
			\draw[color=black] (1.2,.6) node {$e_1$};
			\draw[color=black] (3.73,.6) node {$e_2$};
			\draw[color=black] (6.28,.6) node {$e_3$};
			\draw[color=black] (8.7,.6) node {$e_4$};
			\draw [fill=black] (0.048,0.0) circle (1.5pt);
			\draw [fill=black] (2.59,0) circle (1.5pt);
			\draw [fill=black] (5.16,0.0) circle (1.5pt);
			\draw [fill=black] (7.7,0.0) circle (1.5pt);
			\draw [fill=black] (10.24,0.0) circle (1.5pt);
			
			\draw (0.048,0) parabola (-2,-2);\draw[color=black] (-.8,-.8) node {$\ell-3$};
			\draw (10.24,0)parabola  (12.9,2);\draw[color=black] (11.2,.6) node {$\ell-3$};
			\draw (5.16,0)parabola  (7.0,-2);\draw[color=black] (7,-1) node {$\ell-1$};
			\end{scriptsize}
			\end{tikzpicture}
		\end{center}   
		\caption{Situation (ii) from Theorem~\ref{k>=6 P_{m-1} final count theorem} when $m=2\ell.$}\label{saturationNumberEveni}
	\end{figure}
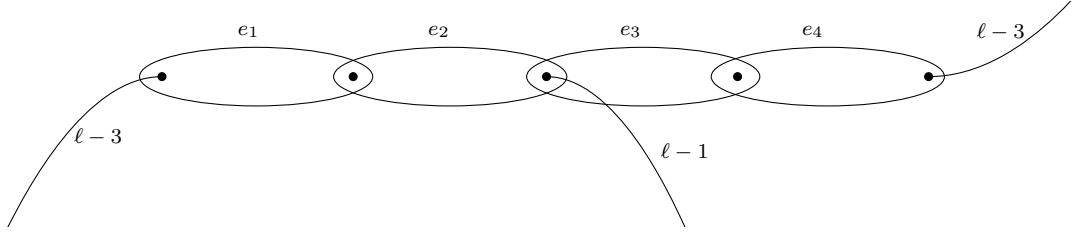
	
	If instead this path attaches to a vertex in $e_3\setminus e_2$, we can apply Lemma~\ref{k>= 6 P_{m-1} branching lemma} on edges $e_1,e_2$ and $e_3$, and similarly to before either we are in situation (i) or we find another path of length at least $\ell-3$ attaching to a vertex in $e_2$. Assuming that we are not in situation (i), by ($\star$), we can assume that the path attaching to a vertex in $e_3\setminus e_2$ and the path attaching to a vertex in $e_2$ are of length $\ell-2$ so that these edges are in some Berge-$P_{m-1}$.
	
	Then, including the paths containing the edges $e_1$ and $e_4$, there are at least four paths leaving vertices in $e_2\cup e_3$, all of length $\ell-2$ in a Berge-$P_{m-1}$. Call this situation (iii) (see figure~\ref{saturationNumberEvenii}).
	
	\begin{figure}  [h]
		\begin{center}  
			\begin{tikzpicture}[line cap=round,line join=round,>=triangle 45,x=1.0cm,y=1.0cm]
			\clip(-2,-2) rectangle (13,1);
			\draw [rotate around={0.:(1.3,0.0)}] (1.30,0.0) ellipse (1.55cm and 0.39cm); 
			\draw [rotate around={0.:(3.88,0)}] (3.88,0.0) ellipse (1.55cm and 0.39cm); 
			\draw [rotate around={0.:(6.446,0.0)}] (6.446,0.0) ellipse (1.55cm and 0.39cm); 
			\draw [rotate around={0.:(8.9,0.0)}] (8.9,0.0) ellipse (1.55cm and 0.39cm); 
			\begin{scriptsize}
			\draw[color=black] (1.2,.6) node {$e_1$};
			\draw[color=black] (3.73,.6) node {$e_2$};
			\draw[color=black] (6.28,.6) node {$e_3$};
			\draw[color=black] (8.7,.6) node {$e_4$};
			\draw [fill=black] (0.048,0.0) circle (1.5pt);
			\draw [fill=black] (2.59,0) circle (1.5pt);
			\draw [fill=black] (5.16,0.0) circle (1.5pt);
			\draw [fill=black] (7.7,0.0) circle (1.5pt);
			\draw [fill=black] (10.24,0.0) circle (1.5pt);
			\draw [fill=black] (3.7,0.0) circle (1.5pt);
			\draw [fill=black] (6.6,0.0) circle (1.5pt);
			
			\draw (0.048,0) parabola (-2,-2);\draw[color=black] (-.8,-.8) node {$\ell-3$};
			\draw (10.24,0)parabola  (12.9,2);\draw[color=black] (11.2,.6) node {$\ell-3$};
			\draw (3.7,0)parabola  (3.2,-2);\draw[color=black] (3.9,-1) node {$\ell-2$};
			\draw (6.6,0)parabola  (7.0,-2);\draw[color=black] (7.4,-1) node {$\ell-2$};
			\end{scriptsize}
			\end{tikzpicture}
		\end{center}   
		\caption{Situation (iii) from Theorem~\ref{k>=6 P_{m-1} final count theorem} when $m=2\ell.$}\label{saturationNumberEvenii}
	\end{figure}
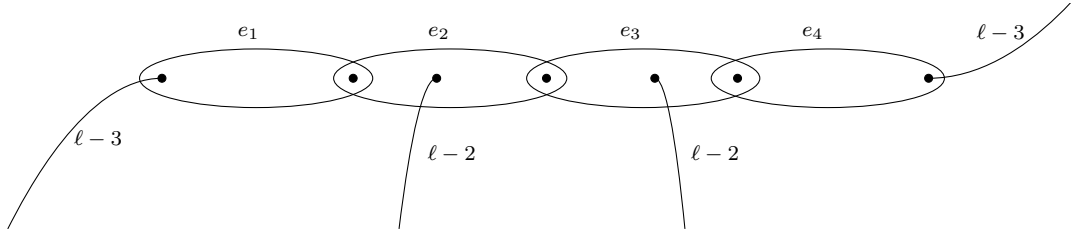

These three situations exhaust all possibilities. Now for even $m=4s+r$, we can use Lemma~\ref{k>= 6 P_{m-1} counting lemma} to count the edges contributed by the paths in each of these situations. The counts are summarized in the table below for even $r$:
	\begin{center}
	{\small\setlength{\extrarowheight}{2pt}\begin{tabular}{|c|c|c|c|}
			\hline
			r & Situation & Calculation &Edge Count\\
			\hline 
			0 & i & $(k-1)(2^{s-1}+2^{s-3}-2)+3$ &$(k-1)(2^{s-1}+2^{s-3})-2k+5$\\
			0 & ii & $3(2^{s}+2^{s-2}-2)$ &$2^{s+1}+2^{s}+2^{s-1}+2^{s-2}-6$\\
			0 & iii & $4(2^s-2)+2$ &$2^{s+2}-6$\\
			\hline
			2 & i & $(k-1)(2^{s}-2)+3$ &$(k-1)(2^{s})-2k+5$\\
			2 & ii & $3(2^{s+1}-2)+4$ &$2^{s+2}+2^{s+1}-6$\\
			2 & iii & $4(2^{s}+2^{s-2}-2)+2$ &$2^{s+2}+2^{s}-6$\\
			\hline
	\end{tabular}}
\end{center}
Since $k\geq 6$, situation (ii) gives the smallest edge count when $r=0$, and situation (iii) gives the smallest when $r=2$, giving us the result for $m$ even.
	
	Now let us consider when $m=2\ell+1$ is odd. The longest path is of odd length $2\ell-1$. Let $e_1, e_2, e_3, e_4$ and $e_5$ be the five central edges of a longest path appearing in that order. Then there are paths of length $\ell-3$ leaving $e_1$ and $e_5$. If we apply Lemma~\ref{k>= 6 P_{m-1} branching lemma} to edges $e_1, e_2$ and $e_3$, we see that there is either a path of length at least $\ell-3$ leaving $e_2$ or there are $k-1$ paths of length $\ell-3$ leaving $(e_1\cup e_3)\setminus e_2$.
	
	If we have the $k-1$ paths leaving $(e_1\cup e_3)\setminus e_2$, we will call this situation (iv). If we are not in situation (iv), then we have a path of length at least $\ell-3$ leaving $e_2$. Now using Lemma~\ref{k>= 6 P_{m-1} branching lemma} on the edges $e_3$, $e_4$ and $e_5$, we either have $k-1$ paths of length $\ell-3$, as in situation (iv) or one path of length $\ell-3$ leaving $e_4$. 
	
	Assuming we are not in situation (iv), we have a path leaving $e_2$ and a path leaving $e_4$, each of length at least $\ell-3$. If one of these paths attach at the vertex in either $e_2\cap e_3$ or $e_3\cap e_4$, call this situation (v), and assume without loss of generality there is a path attached at $e_2\cap e_3$. By ($\star$), we can assume this path is actually of length $\ell-1$ since otherwise the edges in the path would be in no Berge-$P_{m-1}$. Similarly, we can assume the path that attaches to a vertex in $e_4$ is of length at least $\ell-2$, as in Figure~\ref{saturationNumberOddi}. Thus, in situation (v), we have four paths leaving the two edges $e_3$ and $e_4$, two of length $\ell-1$ and two of length $\ell-2$. 
	
	If we are not in situation (v), then neither of the paths found by the first two uses of Lemma~\ref{k>= 6 P_{m-1} branching lemma} attach in $e_3$. Thus if we apply Lemma~\ref{k>= 6 P_{m-1} branching lemma} to the edges $e_2$, $e_3$ and $e_4$, we get that there must be either a path of length at least $\ell-2\geq \ell-3$ leaving $e_3$, which gives us situation (v), or we get at least $k-1$ paths of length at least $\ell-2\geq \ell-3$ leaving vertices in $(e_2\cup e_4)\setminus e_3$, which gives us situation (iv). Thus these two situations are exhaustive.

	\begin{figure}  [h]
		\begin{center}  
			\begin{tikzpicture}[line cap=round,line join=round,>=triangle 45,x=0.8cm,y=0.8cm]
			\clip(-2,-2) rectangle (15.5,2);
			\draw [rotate around={0.:(1.3,0.0)}] (1.30,0.0) ellipse (1.55cm and 0.39cm); 
			\draw [rotate around={0.:(3.88,0)}] (3.88,0.0) ellipse (1.55cm and 0.39cm); 
			\draw [rotate around={0.:(6.446,0.0)}] (6.446,0.0) ellipse (1.55cm and 0.39cm); 
			\draw [rotate around={0.:(8.9,0.0)}] (8.9,0.0) ellipse (1.55cm and 0.39cm); 
			\draw [rotate around={0.:(11.48,0.0)}] (11.48,0.0) ellipse (1.55cm and 0.39cm); 
			\begin{scriptsize}
			\draw[color=black] (1.2,.6) node {$e_1$};
			\draw[color=black] (3.9,.6) node {$e_2$};
			\draw[color=black] (6.28,.6) node {$e_3$};
			\draw[color=black] (8.7,.6) node {$e_4$};
			\draw[color=black] (11.4,.6) node {$e_5$};
			\draw [fill=black] (0.048,0.0) circle (1.5pt);
			\draw [fill=black] (2.59,0) circle (1.5pt);
			\draw [fill=black] (5.16,0.0) circle (1.5pt);
			\draw [fill=black] (7.7,0.0) circle (1.5pt);
			\draw [fill=black] (10.24,0.0) circle (1.5pt);
			\draw [fill=black] (12.78,0.0) circle (1.5pt);
			\draw [fill=black] (9,0.0) circle (1.5pt);

			\draw (0.048,0) parabola (-2,-2);\draw[color=black] (-.8,-.8) node {$\ell-3$};
			\draw (12.78,0)parabola  (15.44,2);\draw[color=black] (13.74,.6) node {$\ell-3$};
			\draw (5.16,0)parabola  (4,2);\draw[color=black] (4.6,1.5) node {$\ell-1$};
			\draw (9,0)parabola  (10,2);\draw[color=black] (9.4,1.5) node {$\ell-2$};
			\end{scriptsize}
			\end{tikzpicture}
		\end{center}   
		\caption{Situation (v) from Theorem~\ref{k>=6 P_{m-1} final count theorem} when $m=2\ell+1.$}\label{saturationNumberOddi}
	\end{figure}
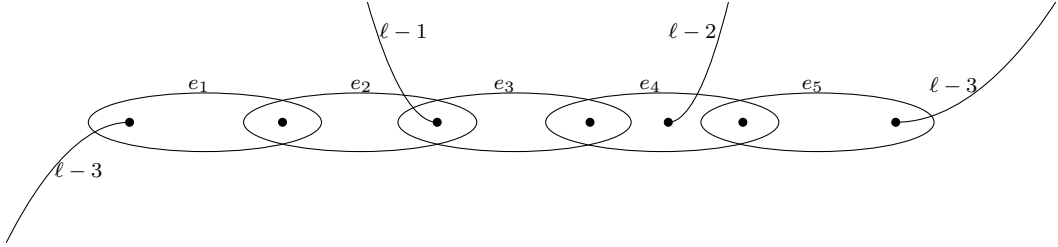
	Just as in the even case, we now use Lemma~\ref{k>= 6 P_{m-1} counting lemma} to count the edges contributed by each path in each situation for odd $m=4s+r$. This is summarized below:
	
		\begin{center}
		{\small\setlength{\extrarowheight}{2pt}\begin{tabular}{|c|c|c|c|}
				\hline
				r & Situation & Calculation &Edge Count\\
				\hline 
				1 & iv & $(k-1)(2^{s-1}+2^{s-3}-2)+3$ &$(k-1)(2^{s-1}+2^{s-3})-2k+5$\\
				1 & v & $2(2^{s}+2^{s-2}-2)+2(2^s-2)+2$ &$2^{s+2}+2^{s-1}-6$\\
				\hline
				3 & iv & $(k-1)(2^{s}-2)+3$ &$(k-1)(2^{s})-2k+5$\\
				3 & v & $2(2^{s+1}-2)+2(2^s+2^{s-2}-2)+2$ &$2^{s+2}+2^{s+1}+2^{s-1}-6$\\
				\hline
		\end{tabular}}
	\end{center}
	Since $k\geq 6$, situation (v) gives the lowest bound for both $r=1$ and $r=3$, completing the proof.
\end{proof}

\section{Saturation Numbers for Berge-$\mathbf{K_3}$, $\mathbf{C_m}$, $\mathbf{K_{1,m}}$, and $\mathbf{\ell K_2}$}\label{section other saturation numbers}
In this section we explore bounds on the Berge saturation numbers for many common classes of graphs. We begin with the saturation number for Berge matchings, $\ell K_2$.

\begin{theorem} \label{thm:matchings}
	If $k\geq 3$, then $\sat_k(n, \text{Berge-}\ell K_2)=\ell-1$ for $n\ge k(\ell-1)$. 
\end{theorem}

\begin{proof}
	Let $H$ be a $k$-uniform Berge-$\ell K_2$-saturated hypergraph with $k\geq 3$ and $n\geq k(\ell-1)$. First, note that any Berge-$\ell K_2$-saturated hypergraph must have at least $\ell-1$ edges so $|E(H)|\geq \ell-1$. 
	
	Now suppose that $H$ is the hypergraph consisting of $\ell-1$ disjoint edges and $n-(\ell-1)k$ isolated vertices. Observe that for any edge $e\in E(\overline{H})$, we can find two vertices $x,y\in e$ such that $x$ and $y$ are not adjacent in $H$. Then every edge in $H$ contains a pair of vertices that are not $x$ or $y$, so every edge in $H$ and $e$ can be used to create a copy of $K_2$ disjoint from all the other copies of $K_2$. Thus $H$ is Berge-$\ell K_2$-saturated.
\end{proof}

Next, we will establish the saturation number for triangles. To do so, we need a definition. A connected component is \emph{edge-minimal} if no connected hypergraph on the same number of vertices has fewer edges. Observe that an edge minimal component with $n'$ vertices has  $\lceil\frac{n'-1}{k-1}\rceil$ edges. Now, the following lemma will be useful in the proof of $\sat_k(n,\text{Berge-}K_3)$.

\begin{lemma}\label{noMinDiv}
Let $k\geq 3$. Let $H$ be a Berge-$K_3$-saturated $k$-uniform hypergraph. If $H$ contains more than one component, then $H$ does not have an edge-minimal component with $n'>1$ vertices such that $(k-1)\mid (n'-1)$.
\end{lemma}

\begin{proof}
Assume $H$ contains an edge-minimal component, $H_1$, with $n'>1$ vertices such that $(k-1)\mid (n'-1)$. 
Then, since $H_1$ is connected with $\frac{n'-1}{k-1}$ edges, $H_1$ is a linear tree.
Thus, $H_1$ has at least one edge containing exactly $k-1$ vertices of degree $1$, say $v_1, v_2,\dots, v_{k-1}$. 
Let $x$ be a vertex in another component of $H$. 
Then, adding the edge $v_1 v_2 \dots v_{k-1}x$ does not introduce a $K_3$. Thus, $H$ is not Berge-$K_3$-saturated.
\end{proof}

\begin{theorem}\label{thm:triangle}
	Let $k\geq 3$. For all $n\ge k+1$, $\sat_k(n,\text{Berge-}K_3)=\CL{\frac{n-1}{k-1}}$.
\end{theorem}

\begin{proof}
	
	Let $S=S_n^{(k)}$ be the $k$-uniform hypergraph with $\left\lfloor\frac{n-1}{k-1}\right\rfloor$ edges that intersect only at a single vertex $v$, and if $r=(n-1)\mod(k-1)>0$, one more edge containing the remaining $r$ vertices, the vertex $v$, and $k-1-r$ other vertices, all from exactly one other edge of $S$ (see Figure \ref{starlike}). 
	Then $S$ has $\CL{\frac{n-1}{k-1}}$ edges and is Berge-$K_3$-free since a Berge-$K_3$ must have three edges, each of which contain two vertices of degree $2$, but $S$ only has two such edges. Further $S$ is Berge-$K_3$-saturated. Indeed, every edge $e\in E(\overline{S})$ shares vertices that are not $v$ with at least two distinct edges of $S$, say $f$ and $g$, so $S+e$ contains the Berge-$K_3$ with edges $e$, $f$ and $g$. Hence, $\sat(n,\text{Berge-}K_3)\leq\CL{\frac{n-1}{k-1}}$.
	
	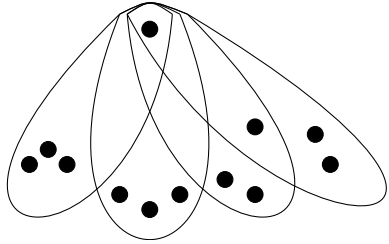
\begin{figure}[h]
		\begin{center}
			\begin{tikzpicture}[scale=2]
			\clip(-1,-1.5) rectangle (1.6,.5);
			
			\draw (-.2,0) .. controls (-2,-1.75) and (0,-1.85) .. (.15,0);
			\draw (-.2,0) .. controls (0,.1) .. (.15,0);
			\draw (-.2,0) .. controls (-1,-2) and (1,-2) .. (.2,0); 
			\draw (-.2,0) .. controls (0,.1) .. (.2,0);
			\draw (-.15,0) .. controls (0,-1.85) and (2,-1.75) .. (.25,0);
			\draw (-.15,0) .. controls (0,.1) .. (.25,0);
			\draw (-.15,0) .. controls (.75,-1.75) and (3,-1.65) .. (.25,0);
			\draw (-.15,0) .. controls (0,.1) .. (.25,0);
			
			\draw [fill=black] (0,-0.1) circle (1.5pt);
			\draw [fill=black] (-.675,-.9) circle (1.5pt);
			\draw [fill=black] (-.8,-1) circle (1.5pt);
			\draw [fill=black] (-.55,-1) circle (1.5pt);
			\draw [fill=black] (0,-1.3) circle (1.5pt);
			\draw [fill=black] (-.2,-1.2) circle (1.5pt);
			\draw [fill=black] (.2,-1.2) circle (1.5pt);
			\draw [fill=black] (.7,-1.2) circle (1.5pt);
			\draw [fill=black] (.5,-1.1) circle (1.5pt);
			\draw [fill=black] (.7,-.75) circle (1.5pt);
			\draw [fill=black] (1.1,-.8) circle (1.5pt);
			\draw [fill=black] (1.2,-1) circle (1.5pt);

			\end{tikzpicture}
			\caption{A copy of $S_{12}^{(4)}$, from Theorem~\ref{thm:triangle} }
			\label{starlike}
		\end{center}
	\end{figure}
	
	Now, let $H$ be a minimal Berge-$K_3$-saturated $k$-graph on $n$ vertices. 
	If $H$ is connected, then $H$ has at least $\CL{\frac{n-1}{k-1}}$ edges and we are done.
	Suppose now that $H$ has $j\geq 2$ components $H_1,\dots,H_j$, with $H_i$ on $n_i$ vertices for $1\leq i\leq j$. 
	By Lemma \ref{noMinDiv}, $H$ does not contain any edge minimal components $H_i$ such that $(k-1) | (n_i-1)$. 
	
	By re-indexing the $H_i$'s if necessary, we can assume that $H$ has $\ell$ non-edge-minimal components where $(k-1) | (n_i-1)$, $H_1,\dots,H_\ell$, that $H$ has $s$ non-edge-minimal components where $(k-1) \nmid (n_i-1)$ $H_{\ell+1},\dots,H_{\ell+s}$, that $H$ has $t$ isolated vertices, $H_{\ell+s+1},\dots,H_{\ell+s+t}$, and that $H$ has $j-\ell-s-t$ edge-minimal components where $(k-1) \nmid (n_i-1)$, $H_{\ell+s+t+1},\dots,H_j$.
	
	Now we introduce a parameter that will help us compare the number of edges in $H$ to the number of edges in $S_n^{(k)}$. For a $k$-uniform hypergraph $G$, let  $O_{G}:=\sum_{v\in V(G)}(d(v)-1)$. This parameter will count the number of times a vertex is covered by more than one edge. Note that $O_G=O_{S_{|V(G)|}^{(k)}} \mod k$ since $O_G=k|E(G)|-|V(G)|$ and $O_{S_{|V(G)|}^{(k)}}=k|E(S_{|V(G)|}^{(k)})|-|V(G)|$. Also note that for our construction $S_n^{(k)}$ above, we have $O_{S_n^{(k)}}\leq \left\lceil\frac{n-1}{k-1}\right\rceil+(k-2)-[(n-1)\mod (k-1)]$.

We will calculate a bound for $O_{H_i}$ for each connected component $H_i$ individually based on the following procedure: Let $E^*\subseteq E(H_i)$ be a smallest set of edges such that the hypergraph $H^*=(V(H_i),E^*)$  is connected. There is an ordering of the edge of $E^*$ such that each edge except the first intersects some preceding edge at least once. Thus, each edge except one in this ordering must contribute at least one to the parameter $O_{H_i}$, and in fact if $(k-1)\nmid (n_i-1)$, the edges of $E^*$ must give at least an additional contribution of $k-1-[(n_i-1) \mod (k-1)]$ since some edges must overlap preceding edges in more than one vertex. Finally, the edges in $E(H_i)\setminus E^*$ will contribute $k$ to this parameter for each edge.
	
For the components $H_i$ of $H$ we get the following bounds on $O_{H_i}$:
{\small
\[
O_{H_i}\geq\begin{cases} 
      \left\lceil\frac{n_i-1}{k-1}\right\rceil -1+k & \text{if }  1\leq i\leq \ell \text{ {\footnotesize (non-minimal, divisible)}}\\[0.3em]
       \left\lceil\frac{n_i-1}{k-1}\right\rceil +k-2-[(n_i-1)\! \!\mod (k-1)]+k & \text{if }\ell+1\leq i\leq \ell +s\text{ {\footnotesize (non-minimal, non-divisible)}}\\[0.3em]
            -1 & \text{if } \ell+s+1\leq i\leq \ell+s+t  \text{ {\footnotesize(isolated vertices)}}\\[0.3em]
       \left\lceil\frac{n_i-1}{k-1}\right\rceil +k-2-[(n_i-1)\!\! \mod (k-1)] & \text{if } \ell+s+t+1\leq i\leq j\text{ {\footnotesize (minimal, non-divisible).}}\\
   \end{cases}
   \]}
Let $r_i=(n_i-1) \mod (k-1)$ for $1\leq i\leq j$ and let $\sum_{i=1}^jr_i=\alpha(k-1)+r^*$ for some $\alpha$ and $r^*=\left(\sum_{i=1}^jr_i\right)\mod (k-1).$
Note that the average remainder of a component that is not an isolated vertex is $\frac{1}{j-t}\sum_{i=1}^jr_i\geq \frac{\alpha(k-1)}{j-t}$, but any such remainder is less than $k-1$, so $\alpha<j-t$.

Then, we have
\[
\begin{array}{ccl}
\sum_{i=1}^jO_{H_i} & \geq & \sum_{i=1}^{\ell}\left( \left\lceil\frac{n_i-1}{k-1}\right\rceil -1+k\right)+\sum_{i=\ell+1}^{\ell+s}\left( \left\lceil\frac{n_i-1}{k-1}\right\rceil +k-2-r_i+k \right)\\
& & +\sum_{i=\ell+s+1}^{j-t}\left(\left\lceil\frac{n_i-1}{k-1}\right\rceil +k-2-r_i\right)-t\\[0.5em]
 & = & \sum_{i=1}^{j}\left\lceil\frac{n_i{-}1}{k{-}1}\right\rceil +\ell(k{-}1)+s(k{-}2{+}k)+(j{-}t{-}\ell{-}s)(k{-}2)-t-\sum_{i=1}^{j}r_i\\[0.5em]
 & \geq & \left\lceil\frac{n-j}{k-1}\right\rceil-\ell-t +(\ell+s)k+(j-\ell-t)(k-2)-\sum_{i=1}^{j}r_i\\[0.5em]
 & \geq & \left\lceil\frac{n-1}{k-1}\right\rceil-1+k-2 -\left(\sum_{i=1}^{j}r_i+j-1-\alpha(k-1)\right)+j-1-\alpha(k-1)\\ 
& & +(j-1)(k-2)+sk+\ell-t(k-1)\\[0.5em]
& \geq & O_{S_n^{(k)}} +(j-\alpha-t)(k-1)+\ell+k(s-1)\\[0.5em]
& \geq & O_{S_n^{(k)}} +k-1+\ell+k(s-1)\\[0.5em]
& \geq & O_{S_n^{(k)}}-1.
\end{array}
\]
	Furthermore, as $O_H-O_{S_n^{(k)}}=0 \mod k$, and $k>1$, we have $O_H\neq O_{S_n^{(k)}}-1$. 
Hence $O_H\geq O_{S_n^{(k)}}$, and so $E(H)\geq E(S_n^{(k)})$.
\end{proof}

The upper bound for the $\sat_k(n,\text{Berge-}K_3)$ is based on the existence of a saturated construction with the desired number of edges. For many $k$ and $n$ combinations, this construction seems to be unique, but this is not true for some uniformity and size pairs. For example, for $k=3$ and $n=8$ there is a second construction, a linear cycle, which also uses $\left\lceil \frac{8-1}{3-1}\right\rceil=4$ edges, see Figure \ref{k3n8}.

\begin{figure}[h]
	\begin{center}
		\begin{tikzpicture}[scale=2]
		\clip(-1,-1.5) rectangle (5,.5);
		\draw (-.2,0) .. controls (-2,-1.75) and (0,-1.85) .. (.15,0);
		\draw (-.2,0) .. controls (0,.1) .. (.15,0);
		\draw (-.2,0) .. controls (-1,-2) and (1,-2) .. (.2,0); 
		\draw (-.2,0) .. controls (0,.1) .. (.2,0);
		\draw (-.15,0) .. controls (0,-1.85) and (2,-1.75) .. (.25,0);
		\draw (-.15,0) .. controls (0,.1) .. (.25,0);
		\draw (-.15,0) .. controls (.75,-1.75) and (3,-1.65) .. (.25,0);
		\draw (-.15,0) .. controls (0,.1) .. (.25,0);
		
		\draw [fill=black] (0,-0.1) circle (1.5pt);
		\draw [fill=black] (-.8,-1) circle (1.5pt);
		\draw [fill=black] (-.55,-1) circle (1.5pt);
		\draw [fill=black] (-.15,-1.2) circle (1.5pt);
		\draw [fill=black] (.15,-1.2) circle (1.5pt);
		\draw [fill=black] (.7,-1.2) circle (1.5pt);
		\draw [fill=black] (.7,-.75) circle (1.5pt);
		\draw [fill=black] (1.2,-1) circle (1.5pt);
		
		\draw (4,0) circle [x radius=.75cm, y radius=.15cm, rotate=0];
		\draw (4,-1) circle [x radius=.75cm, y radius=.15cm, rotate=0];
		\draw (3.5,-.5) circle [x radius=.75cm, y radius=.15cm, rotate=90];
		\draw (4.5,-.5) circle [x radius=.75cm, y radius=.15cm, rotate=90];
		
		\draw [fill=black] (4,0) circle (1.5pt);
		\draw [fill=black] (4,-1) circle (1.5pt);
		\draw [fill=black] (3.5,-.5) circle (1.5pt);
		\draw [fill=black] (4.5,-.5) circle (1.5pt);
		\draw [fill=black] (3.5,0) circle (1.5pt);
		\draw [fill=black] (3.5,-1) circle (1.5pt);
		\draw [fill=black] (4.5,0) circle (1.5pt);
		\draw [fill=black] (4.5,-1) circle (1.5pt);

		\end{tikzpicture}
		\caption{Minimal Berge-$K_3$ saturated hypergraphs on $8$ vertices with uniformity $3$.}
		\label{k3n8}
	\end{center}
\end{figure}
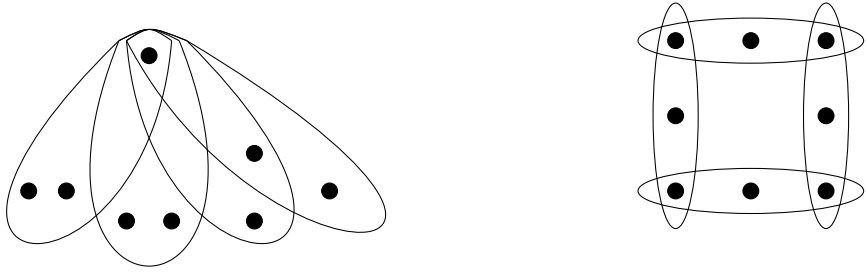

We now turn our attention to an upper bound for cycles. We present three constructions based on the relationship between $k$ and $m$, then prove that these constructions are indeed saturated.

\begin{construction}\label{Cycle construction 1}
	If $k\geq m-1$, let $S_{n,m}^k$ be a $k$-uniform hypergraph containing $\FL{\frac{n-(m-2)}{k-(m-2)}}$ hyperedges that pairwise intersect in a set $I$ of $m-2$ vertices and
	the remaining $((n-(m-2))\mod{(k-(m-2))})$ vertices are in an edge that contains $I$, and some vertices not in $I$ from exactly one edge.
\end{construction}

It is worth noting here that $S_{n,3}^k=S_{n}^k$ from Theorem \ref{thm:triangle}, so our work for cycles can be thought of as a generalization of our work done on triangles.

\begin{construction}\label{cycle construction 2}
	Let $k=m-2$ and $n\geq m^2$. Let $F_{n,m}^{(m-2)}$ be the $k$-uniform hypergraph obtained by first identifying $\FL{\frac{n-1}{m-2}}$ cliques of $m-1$ vertices at a single vertex $v$. If $(m-2)\nmid (n-1)$, then we will absorb the remaining vertices into one clique. Let $r=(n-1)\mod (m-2)$ and let $K$ be one of the $(m-1)$-cliques. Let $x\neq v$ be some vertex in $K$. Add $\left\lceil\frac{r}{k-2}\right\rceil$ edges, $e_1,\dots,e_r$, that contain both $x$ and $v$ so that the $r$ remaining vertices are each in exactly one of these edges, and the edges only contain these $r$ vertices and vertices from $K$. 
\end{construction}

\begin{construction}\label{cycle construction 3}
	Let $k\leq m-3$, $\ell=\max\{m/2+1,k+1\}$ and $n\geq \ell^2$. Let $F_{n,m}^{(k)}$ be the $k$-uniform hypergraph obtained by identifying $\FL{\frac{n-1}{\ell-1}}$ cliques of $\ell$ or $\ell+1$ vertices at a single vertex $v$. There are exactly $(n-1)\mod{(\ell-1)}$ cliques on $\ell+1$ vertices.
\end{construction}

In Construction \ref{cycle construction 3}, we could have chosen any $\ell\in [m/2+1,m-1]$ as long as $\ell\geq k+1$ and gotten a saturated construction, but minimizing $\ell$ under these constraints gives us the smallest number of edges.

\begin{theorem}\label{thm:cycle}
	Let $m\geq 4$. If $k\geq m-1$ and $n>m(k-(m-2))+(m-2)$, then
	\begin{equation}\label{cycle equation 1}
	\sat_k(n,\text{Berge-}C_m)\leq \CL{\frac{n-m+2}{k-m+2}}.
	\end{equation}
	If $k=m-2$ and $n\geq m^2$,
	\begin{equation}\label{cycle equation 2}
	\sat_k(n,\text{Berge-}C_m)\leq \FL{\frac{n-1}{m-2}}\binom{m-1}{k}+\frac{(n-1)\mod{(m-2)}}{k-2}.
	\end{equation}
	If $k\leq m-3$, $\ell=\max\{m/2+1,k+1\}$ and $n\geq \ell^2$, then
	\begin{equation}\label{cycle equation 3}
	\sat_k(n,\text{Berge-}C_m)\leq \FL{\frac{n-1}{\ell-1}}\binom{\ell}{k}+\left((n-1)\mod{(\ell-1)}\right)\binom{\ell}{k-1}.
	\end{equation}
\end{theorem}

\begin{proof}
	First we will prove \eqref{cycle equation 1}. Assume $k\geq m-1$ and $n>m(k-(m-2))+(m-2)$.  We need to show that $S=S_{n,m}^{(k)}$ from Construction \ref{Cycle construction 1} is indeed Berge-$C_m$-saturated. If $(k-m+2)\nmid(n-m+2)$, let $e_1$ and $e_2$ be the two edges of $S$ that intersect outside of $I$. To see that $S$ contains no Berge-$C_m$, first note that the only vertices of degree at least $2$ are in either $I$ or $e_1\cap e_2$. Since $|I|=m-2$, any Berge-$C_m$ would need to use at least two vertices from $e_1\cap e_2$, but there are only two edges incident to any such pair of vertices, so they could not both be used as the degree $2$ vertices in the same Berge-$C_m$.
	
	Let $e\in E(\overline{S})$. $e$ must contain at least two vertices outside of $I$. Let $x$ be one such vertex, and let $e'$ be an edge of $S$ such that $x\in e'$. Note that $e\setminus I$ cannot be contained in $e'$ since $I\subseteq e'$, and so that would imply that $e=e'$. Thus, there is some $y\in e\setminus I$ such that $y\not \in e'$. Let $e''$ be some edge that contains $y$. Then by our choice of $n$, there are enough edges incident with $I$ to create a Berge-$P_m$ in $S$ that goes from $x$ to $y$, using all the vertices of $I$ in the middle. With the addition of $e$, this path can be extended to a Berge-$C_m$ in $S+e$. Thus $S$ is saturated. This proves \eqref{cycle equation 1}.
	
	Now let us focus on \eqref{cycle equation 2}. Let $k=m-2$ and $n\geq m^2$. Let $F=F_{n,m}^{(m-2)}$, $v$, $r$, $K$, $x$ and $e_1,\dots,e_r$ be as in Construction \ref{cycle construction 2}. Each clique has $m-1$ vertices in it, so no clique contains a Berge-$C_m$, and since $v$ is the only vertex incident with more than one clique, no cycles use vertices from more than one clique. Finally, the $r$ vertices not in any $(m-1)$-clique are all of degree $1$, so they cannot be used in a cycle, so $F$ does not contain a Berge-$C_m$.
	
	Now let $e\in E(\overline{F})$. If $e$ contains two vertices from different $(m-1)$-cliques, say $u_1$ and $u_2$, then there is a Berge-$P_m$ from $x$ to $y$ in $F$ that uses as many vertices in each $(m-1)$-clique as needed. We can then use $e$ to close up the cycle, giving us a Berge-$C_m$ in $F+e$. This also works if we have one vertex from some $(m-1)$-clique that is not $K$, and one of the $r$ vertices in no $(m-1)$-clique. If $e$ contains a vertex in $K$, say $u_1$ and one of the $r$ vertices, say $u_2\in e_1$, in no $\ell$-clique, then there is a Berge $P_{m-1}$ from $u_1$ to $x$ inside $K$, and we can use $e_1$ to extend this from $x$ to $u_1$, so there is a $F+e$ contains a Berge-$C_m$. Finally, if $e$ contains only vertices from the $r$ vertices that are not in any $(m-1)$-clique, $e$ must contain two vertices from two different edges, say $u_1\in e_1$ and $u_2\in e_2$. In this case, we can find a Berge-$P_{m-2}$ in $K$ from $v$ to $x$, then extend each end of this path using $e_1$ and $e_2$ to build a Berge-$P_m$ that goes from $u_1$ to $u_2$ in $F$. Thus, $e$ can be used to close up the cycle, so $F+e$ contains a Berge-$C_m$. The only other possibility for $e$ is that $e$ is contained in some $(m-1)$-clique, but all those edges are already present in $F$, so we are done.
	
	Finally, we will show \eqref{cycle equation 3}. Let $k\leq m-3$, $\ell=\max\{m/2+1,k+1\}$ and $n\geq \ell^2$. Let $F=F_{n,m}^{(k)}$ and $v$ be as in Construction \ref{cycle construction 3}. First note that since $k\leq m-3$, $\ell\leq m-2$. Thus, the $\ell$-cliques and the $(\ell+1)$-cliques do not have enough vertices to contain a Berge-$C_m$. Further, no cycle can use vertices from two different cliques since $v$ is the only vertex incident with more than one clique. Thus $F$ does not contain a Berge-$C_m$.
	
	Now, let $e\in E(\overline{F})$. $e$ must contain two vertices from two different cliques, say $x$ and $y$. Since $\ell\geq m/2+1$, and thus $2\ell-1> m$, using the vertices from these two cliques, we can build a Berge-$P_m$ from $x$ to $y$. We can then use $e$ to complete the cycle, so $F+e$ contains a Berge-$C_m$, and thus $F$ is saturated. 
	
	The final thing to note is that $F$ has the number of edges claimed in \eqref{cycle equation 3} since each of the $\FL{\frac{n-1}{\ell-1}}$ cliques contain $\ell$ vertices, they contain a total of $\FL{\frac{n-1}{\ell-1}}\binom{\ell}{k}$ edges, and additionally $(n-1)\mod{(\ell-1)}$ of these cliques each have one extra vertex, which is in $\binom{\ell}{k-1}$ edges.
\end{proof}

Finally, we turn to the saturation number for stars on $k+2$ vertices. A $k$-uniform tight cycle is a hypergraph whose vertex set has a cyclic ordering such that any $k$ consecutive vertices form a hyperedge, and these are the only hyperedges. Here the construction will be a tight cycle along with $k-1$ isolated vertices.

\begin{theorem}\label{thm:k+1star}
	 If $k\geq 3$, then $\sat_k(n,\text{Berge-}K_{1,k+1})=n-k+1$ for $n\geq k^2$.
\end{theorem}

\begin{proof}

	Let $C_{n-k+1}^{(k)}$ be a $k$-uniform tight cycle on $n-k+1$ vertices. Let $\mathcal{C}$ be the hypergraph on $n$ vertices that has $C_{n-k+1}^{(k)}$ as a component along with $k-1$ isolated vertices. We will first show that $\mathcal{C}$ is saturated. Indeed, given any edge $e\not\in E(\mathcal{C})$, $e$ must contain some vertex $v$ in the tight cycle. Let $\{u_1,\dots,u_{k-1}\}$ be the $k-1$ vertices preceding $v$ in the tight cycle.  We can assume that $e\neq v\cup\{u_1,\dots,u_{k-1}\}$ since this edge is already in $\mathcal{C}$. Let $w\in e$ be some vertex such that $w\neq u_i$ for $1\leq i\leq k-1$. Then $\mathcal{C}+e$ contains a Berge-$K_{1,k+1}$ since we can take the edges containing the pairs $vu_i$ for $1\leq i\leq k-1$, without using the edge $e'$ of $C$ that contains $v$ and the $k-1$ vertices succeeding $v$, the edge containing $vw$ and then the edge containing $vx$ for some $x\in e'$, $x\neq w$. Thus $\mathcal{C}+e$ contains a Berge-$K_{1,k+1}$ for any edge $e\not\in E(\mathcal{C})$. Hence, $\sat_k(n,\text{Berge-}K_{1,k+1})\leq E(\mathcal{C})=n-(k-1)$.
	
	Now, let $H$ be a $k$-uniform Berge-$K_{1,k+1}$-saturated graph.
	If there are at least $k$ vertices all of degree at most $k-1$, then these vertices must all be in a clique together, since otherwise we could add an edge only containing such vertices, which would not create a Berge-$K_{1,k+1}$. Thus, $H$ either consists of a clique of $k$ vertices each with degree at least $1$ and $n-k$ vertices of degree at least $k$, or $H$ has of $\ell<k$ vertices of degree less than $k$ and $n-\ell$ vertices of degree at least $k$. In the former case, if we count degrees, we have that
	\[
	|E(H)|\geq \frac{k(1)+ (n-k)k}k=n-k+1.
	\]
	In the latter case, we have
	\[
	|E(H)|\geq \frac{(n-\ell)k}{k} \geq n-k+1.
	\]
\end{proof}
The following theorem establishes linearity for a larger class of stars than Theorem \ref{thm:k+1star}. However, we do not believe the bound to be tight.

\begin{theorem}\label{thm:mstar}
	If $k\leq m-1$, $\sat_k(n,\text{Berge-}K_{1,m})\leq \CL{\frac{n}{m}}\binom{m}{k}$.
\end{theorem}

\begin{proof}
	Let $\mathcal{K}_n^{(k)}$ be the  $k$-uniform hypergraph that is the disjoint union of $\FL{\frac{n}{m}}$ cliques on $m$ vertices, and a clique on $n\mod m$ vertices if $n\mod m \ge k$, or $n\mod m$ isolated vertices if $n\mod m<k$. 
	Note that the maximum size of the neighborhood (the set of vertices adjacent to a given vertex) of any vertex in $\mathcal{K}_n^{(k)}$ is $m-1$, and so the graph does not contain a Berge-$K_{1,m}$.
	Consider adding an edge $e\not\in E(\mathcal{K}_n^{(k)})$. 
	It must contain at least one vertex from a clique on $m$ vertices, say $v$, and a vertex in a different component, say $u$. 
	The clique on $m$ vertices contains a Berge-$K_{1,m-1}$ with $v$ at the center, so adding $\{u, v\}$ yields a Berge-$K_{1,m}.$
	Since $E(\mathcal{K}_n^{(k)}) \le \CL{\frac{n}{m}}\binom{m}{k}$, we have that $\sat_k(n,\text{Berge-}K_{1,m})=O(n)$ when $m-1\ge k$.
\end{proof}

\section{Future Work}
There are many directions that can be explored involving Berge saturation. The most pressing question in the authors' minds is about the asymptotic growth of Berge saturation numbers. For classical saturation, we have that the saturation numbers grow at most linearly with $n$.

\begin{theorem}[K\'aszonyi and Tuza \cite{KT_minedge_1986}]
	For any fixed finite family of graphs $\mathcal{F}$,
	\[
	\sat(n,\mathcal{F})=O(n).
	\]
\end{theorem}

On the other hand, for $k$-uniform hypergraphs, saturation numbers grow with at most $n^{k-1}$.

\begin{theorem}[Pikhurko \cite{pikhurko1999minimum}]
	For any fixed finite family of $k$-uniform hypergraphs $\mathcal{F}^{(k)}$,
	\[
	\sat(n,\mathcal{F}^{(k)})=O(n^{k-1}).
	\]
\end{theorem}
Even though $k$-uniform Berge saturation is a special case of hypergraph saturation, the authors conjecture that asymptotically it resembles the graph case.

\begin{conjecture}
	For any fixed finite family of graphs $\mathcal{F}$,
	\[
	\sat_k(n,\text{Berge-}\mathcal{F})=O(n).
	\]
\end{conjecture}
This conjecture is supported by the results in this paper and our explorations, as every lower bound we have found is at most linear. The authors have also found a few techniques to show linearity for some special classes of graphs not discussed in this paper, but no bound that works for all graphs and all uniformities.

It would also be interesting to see if Berge saturation exhibits any  monotonicity irregularities similar to those discussed by K\'aszonyi and Tuza in \cite{KT_minedge_1986}.

\subsection*{Acknowledgements}
The authors would like to thank Xiaohan Cheng, Andrzej Dudek, Jessica Fuller and Michael Ferrara for valuable discussions. Furthermore, the authors would like to thank the Graduate Research Workshop for Combinatorics (GRWC2016) for facilitating the collaboration of the authors. 

The research of the fourth author is supported by a generous grant from the Combinatorics Foundation and by the Hungarian National Research, Development and Innovation Office -- NKFIH under the grant K116769.

\clearpage
\bibliographystyle{plain}
\bibliography{bsat}

\begin{thebibliography}{10}

\bibitem{bollobas2008pentagons}
B{\'e}la Bollob{\'a}s and Ervin Gy{\H{o}}ri.
\newblock Pentagons versus triangles.
\newblock {\em Discrete Mathematics}, 308(19):4332--4336, 2008.

\bibitem{davoodi2016erd}
Akbar Davoodi, Ervin Gy{\H{o}}ri, Abhishek Methuku, and Casey Tompkins.
\newblock An {E}rd{\H{o}}s-{G}allai type theorem for hypergraphs.
\newblock {\em arXiv preprint arXiv:1608.03241}, 2016.

\bibitem{ergemlidze2017asymptotics}
Beka Ergemlidze, Ervin Gy{\H{o}}ri, and Abhishek Methuku.
\newblock Asymptotics for {T}ur{\'a}n numbers of cycles in 3-uniform linear
  hypergraphs.
\newblock {\em arXiv preprint arXiv:1705.03561}, 2017.

\bibitem{furedi20173}
Zolt{\'a}n F{\"u}redi and Lale {\"O}zkahya.
\newblock On 3-uniform hypergraphs without a cycle of a given length.
\newblock {\em Discrete Applied Mathematics}, 216:582--588, 2017.

\bibitem{gerbner2017asymptotics}
D{\'a}niel Gerbner, Abhishek Methuku, and M{\'a}t{\'e} Vizer.
\newblock Asymptotics for the {T}ur{\'a}n number of {B}erge-{$K_{2,t}$}.
\newblock {\em arXiv preprint arXiv:1705.04134}, 2017.

\bibitem{GP15}
D{\'a}niel Gerbner and Cory Palmer.
\newblock Extremal results for {B}erge-hypergraphs.
\newblock {\em arXiv preprint arXiv:1505.08127}, 2015.

\bibitem{gyHori2006triangle}
Ervin Gy{\H{o}}ri.
\newblock {T}riangle-free hypergraphs.
\newblock {\em Combinatorics, Probability and Computing}, 15(1-2):185--191,
  2006.

\bibitem{GKL10}
Ervin Gy{\H{o}}ri, Gyula~Y. Katona, and Nathan Lemons.
\newblock Hypergraph extensions of the {E}rd{\H{o}}s-{G}allai theorem.
\newblock {\em European Journal of Combinatorics}, 58:238--246, 2016.

\bibitem{gyHori20123}
Ervin Gy{\H{o}}ri and Nathan Lemons.
\newblock 3-uniform hypergraphs avoiding a given odd cycle.
\newblock {\em Combinatorica}, 32(2):187--203, 2012.

\bibitem{EHM64}
Andr{\'a}s Hajnal, Paul Erdos, and J.W. Moon.
\newblock {A Problem in Graph Theory}.
\newblock {\em The American Mathematical Monthly}, 71(10):1107--1110, December
  1964.

\bibitem{KT_minedge_1986}
L.~K\'aszonyi and Zs. Tuza.
\newblock Saturated graphs with minimal number of edges.
\newblock {\em J. Graph Theory}, 10(2):203--210, 1986.

\bibitem{lazebnik2003hypergraphs}
Felix Lazebnik and Jacques Verstra{\"e}te.
\newblock On hypergraphs of girth five.
\newblock {\em {T}he {E}lectronic {J}ournal of {C}ombinatorics}, 10(1):R25,
  2003.

\bibitem{M07}
W.~Mantel.
\newblock Problem 28.
\newblock {\em Wiskundige Opgaven}, 10(60-61):320, 1907.

\bibitem{palmer2017tur}
Cory Palmer, Michael Tait, Craig Timmons, and Adam~Zsolt Wagner.
\newblock Tur{\'a}n numbers for {B}erge-hypergraphs and related extremal
  problems.
\newblock {\em arXiv preprint arXiv:1706.04249}, 2017.

\bibitem{pikhurko1999minimum}
Oleg Pikhurko.
\newblock The minimum size of saturated hypergraphs.
\newblock {\em Combinatorics, Probability and Computing}, 8(5):483--492, 1999.

\bibitem{timmons2016r}
Craig Timmons.
\newblock On $r$-uniform linear hypergraphs with no {B}erge-{$K_{2,t}$}.
\newblock {\em arXiv preprint arXiv:1609.03401}, 2016.

\bibitem{T41}
P{\'a}l Tur{\'a}n.
\newblock Oeine {E}xtremalaufgabe aus der {G}raphentheorie.
\newblock {\em Mat. Fiz. Lapok}, 48(436-452):137, 1941.

\end{thebibliography}

\end{document}